\documentclass{amsart}

\usepackage{todonotes}

\usepackage{amsthm}
\usepackage{amssymb}
\usepackage{amsmath}
\usepackage{xfrac}
\usepackage{array}
\usepackage{mathtools}
\usepackage{stmaryrd}
\usepackage{bbm}
\usepackage[multiple]{footmisc}

\newtheorem{prop}{Proposition}
\newtheorem{definition}[prop]{Definition}
\newtheorem{theorem}[prop]{Theorem}

\newtheorem{thm}[prop]{Theorem}
\newtheorem{cor}[prop]{Corollary}
\newtheorem{lemma}[prop]{Lemma} 
\newtheorem{remark}[prop]{Remark}

\numberwithin{prop}{section}

\newcommand{\Sub}[1]{\ensuremath{\mathrm{Sub}\left( #1 \right) }}

\begin{document}

\title[Stabilizers and critical exponents]{Critical exponents of invariant random subgroups in negative curvature}
\author{Ilya Gekhtman  and Arie Levit}

\begin{abstract}
Let $X$ be a proper geodesic Gromov hyperbolic metric space and let  $G$ be a cocompact group of isometries of $X$ admitting  a uniform lattice.  Let $d$ be the Hausdorff dimension of the Gromov boundary $\partial X$. We define the critical exponent $\delta(\mu)$ of any  discrete invariant random subgroup $\mu$ of the locally compact group $G$   and show that   $\delta(\mu) > \frac{d}{2}$ in general and that $\delta(\mu) = d$ if $\mu$ is of divergence type. 
Whenever  $G$ is a rank-one simple Lie group with Kazhdan's property $(T)$  it follows that  an ergodic invariant random subgroup of divergence type  is  a lattice. One of our main tools is a maximal ergodic theorem for actions of hyperbolic groups due to Bowen and Nevo.
\end{abstract}

\maketitle


\section{Introduction}
\label{sec:introduction}

Let $X$ be a proper geodesic Gromov hyperbolic  metric space  and let $\mathrm{Isom}(X)$ denote its group of isometries. Fix a basepoint $o\in X$.
The critical exponent $\delta(\Gamma)$ of a discrete subgroup $\Gamma$ of $\mathrm{Isom}(X)$ is given by
$$\delta(\Gamma)= \inf \{s:\sum_{\gamma \in \Gamma}e^{-sd_X(o,\gamma o)}<\infty \}.$$
The subgroup $\Gamma$  is   of divergence type if the above series diverges at the critical exponent.
A discrete subgroup  of $\mathrm{Isom}(X)$ is a uniform lattice if it acts cocompactly on $X$. A uniform lattice is always of divergence type and its critical exponent   is equal to the Hausdorff dimension $\dim_\mathrm{H}(\partial X)$   of the Gromov boundary of $X$.

For example, if $X$ is  the $n$-dimensional hyperbolic space $\mathbb{H}^n$ and $\Gamma$ is a uniform lattice in $\mathrm{Isom}(\mathbb{H}^n)$ then $\delta(\Gamma) = \dim_\mathrm{H}(\partial \mathbb{H}^n)  = n-1$.



An invariant random subgroup of a locally compact group $G$ is a conjugation invariant Borel probability measure on the Chabauty space of its closed subgroups. Lattices as well as  normal subgroups of  lattices provide examples of invariant random subgroups. 

\begin{theorem}
\label{thm:main theorem}
Let $G$ be a closed non-elementary subgroup of $\mathrm{Isom}(X)$ acting cocompactly on $X$ and admitting a uniform lattice. 
Let $\mu$ be an  invariant random subgroup of $G$ such that $\mu$-almost every closed subgroup of $G$ is discrete and infinite. Then
\begin{enumerate}
\item $ \delta(H) > \frac{1}{2}\dim_\mathrm{H}(\partial X)$, and 
\item  if $H$ has divergence type then    $\delta(H) = \dim_\mathrm{H}(\partial X)$ 
\end{enumerate}
for $\mu$-almost every closed subgroup $H$.
\end{theorem}

Since $ \delta(\Gamma) = \liminf_{R \to \infty} \frac{1}{R} \ln   |\Gamma o \cap B_X(o,R)|$,  the conclusion of part (1) of Theorem \ref{thm:main theorem}  says that the orbits of  discrete invariant random subgroups of $G$ are in a certain sense "large".



Patterson constructed discrete subgroups of $\mathrm{Isom}(\mathbb{H}^n)$ having full limit sets, arbitrary small  positive critical exponents and    being of  convergence as well as of divergence type \cite{Pat}. By our main result such discrete subgroups cannot be assigned positive measure by any invariant random subgroup of $\mathrm{Isom}(\mathbb{H}^n)$.

In particular  Theorem \ref{thm:main theorem} applies to the action of a word hyperbolic group $G$ on its Cayley graph. This viewpoint is useful and gives a sharper result even if $G$ happens to be   a uniform  lattice in a rank one simple linear group, since in general the exact value of the critical exponent  changes under a quasi-isometry.

If $G$ is a rank one simple linear group over a local field then it admits a uniform lattice and any non-atomic invariant random subgroup of $G$ is discrete and infinite  \cite{7S, GL}, so that these assumptions of Theorem \ref{thm:main theorem}  are  satisfied automatically. 
We obtain the following characterization of invariant random subgroups of divergence type.

\begin{cor}
\label{cor:IRS of divergence type iff geodesic flow is ergodic}
Let $G$ be a rank-one simple Lie group and   $\mu$ be a non-atomic  invariant random subgroup of $G$. Then  $\mu$-almost every discrete subgroup $\Gamma$ of $G$ has divergence type if and only if the geodesic flow on $\Gamma\backslash G$ is  ergodic with respect to the Haar measure class. 
\end{cor}



\subsection*{Kesten's theorem for invariant random subgroups of  rank-one simple Lie groups} 
Consider the  case where $G$ is a rank-one simple Lie group and  $X$ is the associated symmetric space.  Let $\Gamma$ be a discrete torsion-free subgroup of $G$ so that the quotient $\Gamma \backslash X$   is   a locally symmetric Riemannian manifold. 

Let $\lambda_0(\Gamma \backslash X)$ denote the bottom of the spectrum of the Laplace--Beltrami operator on $\Gamma \backslash X$. The two quantities  $\lambda_0(\Gamma \backslash X)$ and $\delta(\Gamma)$ stand in the following quadratic relation \cite{elstrodt1973resolvente, elstrodt1973resolvente2, elstrodt1974resolvente3, patterson1976limit, sullivan1987related, leuzinger2004critical}
$$
\text{$\lambda_0(\Gamma \backslash X) = \begin{cases} 
 \frac{1}{4}d(X) ^2 & \text{if  $ 0 \le \delta(\Gamma) \le  \frac{1}{2}d(X)$} \\ 
\delta(\Gamma)(d(X) - \delta(\Gamma)) & \text{if  $ \frac{1}{2}d(X) \le \delta(\Gamma) \le  d(X)$} \end{cases}$  where $d(X) = \dim_\mathrm{H}(\partial X)$.}
$$

For instance  $\lambda_0(X) = \frac{1}{4}\dim_\mathrm{H}(\partial X)^2$. Moreover $\Gamma$ has critical exponent  equal to $\dim_\mathrm{H}(\partial X)$ if and only if $\Gamma$ is co-amenable.

\begin{cor}
\label{cor:Kestens theorem for IRS in surfaces}
Let $G$ be a rank-one simple Lie group and $\mu$ be a non-atomic    invariant random subgroup   of $G$ so that $\mu$-almost every subgroup is torsion-free. Then 
$\lambda_0(\Gamma \backslash X) < \lambda_0(X)$ holds $\mu$-almost surely.
\end{cor}

We know from Theorem \ref{thm:main theorem}  that $\mu$-almost every  subgroup is non-elementary and is in particular non-amenable. Therefore our Corollary  \ref{cor:Kestens theorem for IRS in surfaces}  is  a natural generalization of the main result of Ab\'{e}rt, Glasner and Vir\'{a}g \cite{abert2014kesten} to  invariant random subgroups of  rank one simple Lie groups  with respect to the geometric Laplacian.

It is known that higher-rank symmetric spaces  satisfy Kesten's theorem for irreducible invariant random subgroups in a much stronger sense.   Indeed, every non-atomic irreducible invariant random subgroup $\mu$ in a higher-rank Lie group   is co-amenable \cite{SZ, 7S, HT} and so $\lambda_0(\Gamma \backslash X) = 0$ holds $\mu$-almost surely. Therefore  our Corollary \ref{cor:Kestens theorem for IRS in surfaces} completes the picture for all semisimple Lie groups.




\subsection*{Hyperbolic spaces with Kazhdan isometry groups}

Whenever $\mathrm{Isom}(X)$ has  Kazhdan's property $(T)$ our main result admits an interesting application, which can be regarded as a certain  conditional rank one analogue of the celebrated theorem of Stuck and Zimmer \cite{SZ}.


\begin{cor}
\label{cor:divergence type with T}
If $G$ is a  rank-one simple Lie group with Kazhdan's property (T) then any non-atomic ergodic invariant random subgroup of $G$   of divergence type   is essentially a lattice. 

Similarly, if  $G$ has Kazhdan's property (T) and is either a hyperbolic group acting on its Cayley graph or a uniform lattice in $\mathrm{Isom}(X)$ for some  $\mathrm{CAT}(-1)$-space   $X$ then any   invariant random subgroup of $G$ is supported  on finite index subgroups.
\end{cor}

Corollary \ref{cor:divergence type with T} follows directly from part (2) of Theorem \ref{thm:main theorem} together with Corlette's  work \cite{Cor} on discrete subgroups of rank one Lie groups with property $(T)$  and its generalization  by Coulon--Dal'bo--Sambusetti \cite{CDS}. Indeed, if $G$ has property $(T)$ then there exists a certain threshold value $0 < \delta_c < \dim_\textrm{H}(\partial X)$ such that every discrete subgroup $\Gamma$ with $\delta(\Gamma) > \delta_c$ must be a lattice.





\subsection*{Stabilizers of probability measure preserving actions}

Invariant random subgroups can be used towards the  study of   probability measure preserving actions by considering   stabilizers of   random points. 

\begin{cor}
\label{cor:application to pmp actions}
Let $G$ be a closed non-elementary  subgroup of $\mathrm{Isom}(X)$ acting cocompactly on $X$ and admitting a uniform lattice. Let $(Z,\mu)$ be a Borel space with a probability measure preserving action of $G$.  If the stabilizer $G_z$ of $\mu$-almost every point $z \in Z$ is discrete and infinite then $\mu$-almost surely $\delta(G_z)  > \frac{\dim_\mathrm{H}(\partial X)}{2} $.
\end{cor}


By inducing invariant random subgroups from a lattice we obtain the following.

\begin{cor}
\label{cor:application to pmp actions of discrete groups}
Assume that $\mathrm{Isom}(X)$ is non-elementary and admits a uniform lattice. Let $\Gamma$ be any lattice in $\mathrm{Isom}(X)$ and   $(Z,\mu)$ be any Borel space with a probability measure preserving action of $\Gamma$.  Then  the stabilizer $\Gamma_z$ of $\mu$-almost every point $z \in Z$ is either finite or satisfies $ \delta(\Gamma_z) > \frac{\dim_\mathrm{H}(\partial X)}{2}$.
\end{cor}

We emphasize that  the lattice $\Gamma$ considered in Corollary \ref{cor:application to pmp actions of discrete groups} need not be uniform. 
For example, as an abstract group $\Gamma$ can be  a finitely generated free group or the fundamental  group of any finite volume hyperbolic manifold.

\subsection*{A few words about our proofs}

Let $\Gamma $ be a uniform lattice in $\mathrm{Isom}(X)$ with critical exponent $\delta(\Gamma)$. Consider some fixed hyperbolic element $h \in \mathrm{Isom}(X)$. Our proof relies on the  essentially elementary  observation that a partial Poincare series evaluated over the orbit of $h$ under conjugation by elements of the lattice $\Gamma$ behaves like a Poincare series of exponent  $2\delta$. This strategy relies on  hyperbolicity.  Of course, unless $h$ belongs to a  normal subgroup, there is no reason for  these conjugates to be contained in any particular discrete subgroup. This  complicates our proofs considerably. To overcome this problem we use a  maximal ergodic theorem for hyperbolic groups, due to Bowen--Nevo \cite{Bowen-Nevo-negcurved}, to obtain a quantitative estimate for the recurrence of the invariant random subgroup to a neighborhood of the orbit  in question. Finally we use quasiconformal densities and  Patterson--Sullivan theory to deduce sharper results for divergence type groups. The paper includes an appendix on   measurability issues of the above notions in the  Chabauty topology.

\subsection*{Previous work}

A normal subgroup of a lattice is a very special kind of an invariant random subgroup. Our work extends  results known in that setting. Theorem \ref{thm:main theorem} is originally  due to  Jaerisch \cite{Jaerisch} in the case of normal subgroups of  Kleinian groups. This was later generalized by Matsuzaki--Yabuki--Jaerisch \cite{MYJ} to deal with  normal subgroups of discrete isometry groups of proper Gromov hyperbolic spaces. 
 Arzhantseva and Cashen \cite{Arzh-Cash} have recently generalized    \cite{MYJ} to   normal subgroups of finitely generated groups acting on proper geodesic spaces with a strongly contracting element. This class includes  rank-one actions on $\mathrm{CAT}(0)$ spaces as well as the mapping class group action  on the Teichm\"{u}ller space equipped with the Teichm\"{u}ller metric.

We recall  Kesten's  theorem \cite{kesten1959symmetric}  and its recent generalization by Ab\'{e}rt, Glasner and Vir\'{a}g \cite{abert2014kesten} to invariant random subgroups. Let $\Gamma$ be any finitely generated  group and $S \subset \Gamma$ be some symmetric finite generating set.  
Let $\lambda_0(\Gamma)$ denote  the bottom of the spectrum of the combinatorial Laplacian on the Cayley graph of $\Gamma$ with respect to $S$. Similarly  $\lambda_0(H\backslash \Gamma)$  is  defined   for any quotient of the Cayley graph by a  subgroup $H \le \Gamma$. Given an invariant random subgroup $\mu$ of $\Gamma$ it is shown in \cite{abert2014kesten} that  $\lambda_0(H\backslash \Gamma) < \lambda_0(\Gamma)$ if and only if the subgroup $H$ is non-amenable $\mu$-almost surely. Kesten's classical theorem is the special case where $\mu$ is supported on some normal subgroup $N$ of $\Gamma$. Our Corollary \ref{cor:Kestens theorem for IRS in surfaces} can be seen as a generalization of this, and is based  on completely different methods.

 Cannizzo  showed  that  the action of an invariant random subgroup of a countable group on the Poisson boundary of its enveloping group is almost surely  conservative  \cite[Theorem 1.2.2]{Cannizzo}. Grigorchuk, Kaimanovich and Nagnibeda  showed that if $\Gamma$ is a free group  and 
$H $ is a subgroup so that $\delta(H)<\frac{\delta(\Gamma)}{2}$ with respect to a free generating set, then the action of $H$ on the boundary of   $\Gamma$  is dissipative \cite[Theorem 4.2]{GKN}. The analogous result for discrete Fuchsian groups where the critical exponent taken with respect to the hyperbolic plane metric  is due to  Patterson \cite{Pat2} and Matsuzaki \cite{Mat}.  Combining these  results provides a different proof of a non-strict variant of part (1) of our Theorem \ref{thm:main theorem} for these particular actions of free groups and surface groups.  
 We remark that  \cite{Pat2}  uses spectral methods while the approach of \cite{GKN} is  combinatorial, so that these proofs appear to be difficult to generalize to the general Gromov hyperbolic setting.

\subsection*{Acknowledgements}

We would like to thank Mikolaj Fraczyk for pointing out the analogy with Kesten's theorem and leading us to include Corollary \ref{cor:Kestens theorem for IRS in surfaces}. We would like to thank Amos Nevo for useful remarks and suggestions concerning the  ergodic theorem for hyperbolic groups. We would  like to thank Tushar Das and David Simmons for illuminating remarks about critical exponents, especially alerting us to Patterson's construction of Kleinian groups with full limit set and arbitrarily small critical exponent.
The first author is partially supported by NSF grant DMS-1401875.

\section{The critical exponent in Gromov hyperbolic spaces}


Let $X$ be a proper geodesic Gromov hyperbolic metric space. We recall some basic notions regarding the Gromov boundary, the Poincare series and the critical exponent. 

\subsection*{The Gromov boundary}

A geodesic metric space $X$ is \emph{Gromov  hyperbolic} if all geodesic triangles with endpoints in $X$ are $\delta$-thin for some constant $\delta > 0$.
The  \emph{Gromov boundary} $\partial X$ of  a proper geodesic Gromov hyperbolic space $X$ is  the set of all geodesic rays,   with two rays of finite Hausdorff distance being identified. 

Ideal geodesic triangles with endpoints in $\partial X$ are also $\delta'$-thin for some constant $\delta' > 0$ depending on $\delta$. In particular any two geodesic lines having the same endpoints either in $X$ or in $\partial X$ are within Hausdorff distance $\delta'$ of each other \cite{Ohshika}. Denote
$$\partial^2 X = (\partial X \times \partial X) \setminus \mathrm{Diag} $$
so that $\partial ^2 X$ consists of all ordered pairs of distinct  points on the boundary of $X$. For every $(x^-, x^+) \in \partial^2 X$ there is a bi-infinite geodesic $l$ in $X$ with $l(-\infty)= x^-$ and $l(\infty) = x^+$.


%
%

A \emph{visual metric} $\rho$ on the Gromov boundary $\partial X$ with parameter $a > 0$ and basepoint $o \in X$  is any metric   satisfying
$$ k_1 a^{-d_X(o,l_{\xi_1,\xi_2}(\mathbb{R}))} \le \rho(\xi_1,\xi_2) \le k_2 a^{-d_X(o,l_{\xi_1,\xi_2}(\mathbb{R}))} $$
for some constants $ k_1, k_2 > 0$ and every pair of points $\xi_1,\xi_2 \in \partial X$, where $l_{\xi_1,\xi_2} : \mathbb{R} \to X$ is any geodesic line with end points $l_	{\xi_1,\xi_2}(-\infty) = \xi_1$ and $l_{\xi_1,\xi_2}(\infty) = \xi_2$. A visual metric $\rho_a$ on $\partial X$ exists for any value of the parameter  $a$ within the range $a'> a>1$, where $a' > 1$ is a constant depending only on $\delta$. 

The Hausdorff dimension of the metric $\rho_a$ is equal to $\ln(a) d(X)$ for some constant $d(X)$ independent of the parameter $a$ and the basepoint $ o \in X$. By a slight abuse of notation, the \emph{Hausdorff dimension} of the Gromov boundary $\partial X$ is defined as $\dim_\text{H}(\partial X) = d(X)$. It is known that $\dim_\text{H}(\partial   X)>0$ whenever $|\partial X|>2$ and that $\dim_\text{H}(\partial X) < \infty$ whenever $\textrm{Isom}(X)$ admits a uniform lattice. 

The Hausdorff dimension   of the Gromov boundary of the $n$-dimensional hyperbolic space is equal to its topological dimension $\dim_{\mathrm{top}}(\partial \mathbb{H}^n) = n-1$. We point out that in general the Hausdorff and the topological dimensions  of the Gromov boundary may be different. In fact Bonk and Kleiner \cite{BK} proved that if $X$ is a $\mathrm{CAT}(-1)$-space   not isometric to the $n$-dimensional hyperbolic space $\mathbb{H}^n$ for some $n$  then the strict inequality $d(X)>\dim_{\mathrm{top}} (\partial X)$ holds. 
%
%



%
%
%


The Gromov compactification of the metric space $X$   is denoted $\overline{X}=X \cup \partial X$. The  group of isometries $\mathrm{Isom}(X)$ is locally compact and acts continuously on $\overline{X}$. 

\subsection*{The Poincare series and the critical exponent}

Let $\Gamma$ be any discrete  subgroup of $\mathrm{Isom}(X)$. 
The \emph{Poincare series} of  $\Gamma$ at the exponent $s \ge 0 $ and with respect to the basepoint $o \in X$  is given by
$$ \mathcal{P}_\Gamma(s) = \sum_{\gamma \in \Gamma}e^{-sd_X(o,\gamma o)}. $$
The \emph{critical exponent} of $\Gamma$ is the infimum over all exponents $s$ such that $\mathcal{P}_\Gamma(s)$ converges. It is independent of the choice of basepoint. It is easy to see that
$$ \delta(\Gamma) = \liminf_{R \to \infty} \frac{1}{R} \ln |\Gamma o \cap B_X(o,R)|.$$

The discrete group $\Gamma$ is said to be of \emph{divergence type} if $\mathcal{P}_\Gamma(\delta(\Gamma))$ diverges,  and of \emph{convergence type} otherwise. If $\Gamma$ is a uniform lattice then it is necessarily of divergence type  \cite[Corollary 7.3]{Coornaert}. A discrete group of divergence type is clearly infinite.

The \emph{limit set} $\Lambda(\Gamma)$ is the set of all accumulation points in the Gromov boundary $\partial X$ of some  orbit of $\Gamma$ in $X$. The subgroup $\Gamma$ is \emph{non-elementary} if $\Lambda(\Gamma)$ has more than two points. If $\Gamma$ is non-elementary then $\delta(\Gamma)>0$.

The \emph{radial limit set} $\Lambda_{r}(\Gamma)\subset \Lambda(\Gamma)$ consists of the ideal points $\alpha \in \partial X$ such that there is a geodesic ray  converging to $\alpha$ and intersecting non-trivially infinitely many $\Gamma$-translates of the ball $B_X(o,R)$ for some $R > 0$.
It is known by the work of  Bishop-Jones \cite[Theorem 1]{BJ} in the case of Kleinian groups and of Das-Simmons-Urbanski \cite[Theorem 1.2.1]{DSU} in the case of Gromov hyperbolic spaces that
$$\delta(\Gamma)=\dim_{\textrm{H}}(\Lambda_{r}(\Gamma)).$$ 
The above immediately implies the upper bound $ \delta(\Gamma) \le \dim_\textrm{H}(\partial X)$. If  $\Gamma$ is convex-cocompact, then every limit point is radial so that
$$\dim_{\textrm{H}}\Lambda_{r}(\Gamma)=\dim_{\textrm{H}}\Lambda( \Gamma).$$
In particular, if $\Gamma$ is a uniform lattice in $X$ then $\Lambda (\Gamma) = \partial X$ and
$$\delta(\Gamma)=\dim_{\textrm{H}}\partial X.$$


\subsection*{Partial Poincare series} Let $\Gamma$ be a discrete subgroup of $\mathrm{Isom}(X)$. It  will often be  useful to consider a partial Poincare series, where the summation is performed over a subset of $\Gamma$.
 We are particularly interested in the following situation. Assume that the group $\Gamma$ is acting on a set $Z$. Given a point  $z \in Z$ and a subset $Y \subset Z$ denote 
$$ E_\Gamma(z, Y) = \{ \gamma \in \Gamma \: : \: \gamma z \in Y \}. $$

\begin{definition}
The \emph{partial Poincare series} of $\Gamma$ at the exponent $s \ge 0$ with respect to the given action of $\Gamma$ on $Z$ and the basepoint $o \in X$ is 
$$ \mathcal{P}_\Gamma(s;z,Y) = \sum_{\gamma \in E_\Gamma(z,Y)}e^{-sd_X(o,\gamma o)} $$
where $z \in Z$ is any point and $Y \subset Z$ is any subset.
\end{definition}
Clearly $\mathcal{P}_\Gamma(s;z,Y) \le \mathcal{P}_\Gamma(s)$ always holds, and we will be concerned with establishing useful estimates in the other direction. 

Two different actions of $\Gamma$ will play a role below. Namely, either $Z = \mathrm{Isom}(X)$ and $V$ is an open subset (consisting of hyperbolic elements), or $Z$ is a Borel space with a probability measure $\mu$ and $Y$ is a Borel  subset with $\mu(Y) > 0$.

\section{Hyperbolic elements in the Poincare series}

Let $X$ be a proper geodesic Gromov hyperbolic metric space and $\Gamma$  a fixed uniform lattice  in $\mathrm{Isom}(X)$. 

\subsection*{Hyperbolic elements}

An isometry $h \in \mathrm{Isom}(X)$ is \emph{hyperbolic} if 
$h$ has exactly two fixed points $h^+, h^-$ on the boundary $\partial X$ and $h$ is of infinite order.
The points $h^+$ and $h^-$ are called the 
\emph{attracting} and \emph{repelling} points of the hyperbolic element $h$, respectively. We have that
$$ \text{$h^{ n} x \to h^+$ for all $x \in \partial X \setminus \{h^-\}$ as $n \to \infty$ } $$
as well as the analogous property with $h^+$ replaced by $h^-$ and $h^{n}$ by $h^{-n}$.

Let $\mathcal{H}(X)$ denote the collection of all hyperbolic elements in $\mathrm{Isom}(X)$. The collection $\mathcal{H}(X)$ is clearly invariant under conjugation. There is a natural map
$$ \mathcal{E} : \mathcal{H}(X) \to \partial^2 X, \quad \mathcal{E}(h) = (h^-, h^+).$$
It is easy to see that the map $\mathcal{E}$ is  continuous. 
Moreover   $\mathcal{E}$ is  $\mathrm{Isom}(X)$-equivariant 
 with respect to the conjugation action on $\mathcal{H}(X)$ and the diagonal action on $\partial^2 X $ induced by the embedding $\mathrm{Isom} (X)\to \mathrm{Homeo}(\partial X)$. 
 





\begin{prop}
\label{prop:hyperbolic elements in CAT minus one groups are open}
The collection $\mathcal{H}(X)$ of hyperbolic elements is open in $\mathrm{Isom}(X)$.
\end{prop}
\begin{proof}
Let $h \in \mathrm{Isom}(X)$ be a hyperbolic isometry with attracting and repelling points
$ \mathcal{E}(h) = (h_-, h_+) \in \partial^2 X$.
The compactness of $\partial X$ combined with the contraction property of the hyperbolic isometry $h$ allows us to choose two  disjoint open neighborhoods $h^- \in U^-$ and $h^+ \in U^+$ in $\partial X$  so that $ U^- \cup U^+ \neq \partial X$ and
$$\text{$h^N(\partial X \setminus U^-) \subset U^+$ and $h^{-N}(\partial X \setminus U^+) \subset U^{-}$}$$
for some sufficiently large $N \in \mathbb{N}$.

The above condition defines an open subset of $\mathrm{Homeo}(\partial X)$  in the compact-open topology. Since the natural map $G \to \mathrm{Homeo}(\partial X)$ is continuous, 
this continues to hold for every element $g \in \mathrm{Isom}(X)$ in a sufficiently small neighborhood of $h$. 
We claim that any such $g$ must be hyperbolic as well. 

For  any point $\xi \in \partial X \setminus (U^- \cup U^+) $ the condition 
 $g^{mN} \xi \in U^+ $
holds for all $m\in\mathbb{N}$. In particular the isometry $g^N$, and therefore $g$ itself, is of infinite order. 

It remains to rule out the possibility that the isometry $g$ is parabolic. In that case   $g^N$ is clearly parabolic as well. Note that
$$
\text{$g^{mN}(\partial X \setminus U^-) \subset U^+$ and $g^{-mN}(\partial X \setminus U^+) \subset U^-$} $$ holds for all $m \in \mathbb{N}$.
Suppose $g^N$ is parabolic. Let $p\in \partial X$ be the unique fixed point of $g^N$. For any point $\xi \in \partial X$ we would have both $g^{mN}\xi\to p$ and $g^{-mN}\xi\to p$ as $m\to \infty$. This contradicts the fact that $U^-$ and $U^+$ are disjoint. We conclude that $g$ is hyperbolic.
\end{proof}

\begin{prop}
\label{prop:compact set of hyperbolic elements meets a ball}
Let  $K \subset \mathcal{H}(X) \subset \mathrm{Isom}(X)$ be a compact subset. Then for any point  $o \in X$ there is a radius $D = D(K, o) > 0$ such that any bi-infinite geodesic line $l$ in $X$ with $l(-\infty) = h^-$ and $l(\infty) = h^+$   where $(h^-, h^+) = \mathcal{E}(h)$ for some  $h \in K$ satisfies  $l \cap B_{X}(o,D) \neq \emptyset$. 
\end{prop}
\begin{proof}
 The definition of the visual metric on $\partial X$ implies that for any compact subset $Q \subset \partial^{2}X$ there is a compact $C\subset X$ such that any bi-infinite geodesic line $l$ in $X$ with endpoints $(l(-\infty), l(\infty)) \in Q$ satisfies $l \cap C \neq \emptyset$. The result now follows from the continuity of the map $ \mathcal{E} : \mathcal{H}(X) \to \partial^2 X$.
\end{proof}

\subsection*{Partial Poincare series over conjugates of a hyperbolic element}

Recall that given any element $g \in \mathrm{Isom}(X)$ and subset $U \subset \mathrm{Isom}(X) $ we denote
$$E_\Gamma(g, U) = \{ \gamma \in \Gamma \: : \: \gamma g \gamma^{-1} \in U \}$$ and consider the associated partial Poincare series 
$$\mathcal{P}_\Gamma(s;g,U)=\sum_{\gamma \in E_\Gamma(g,U)}e^{-sd_{X}(\gamma o,o)}.$$

\begin{lemma}
\label{lemma:introducing one half}
Let $h\in \mathcal{H}(X)$ be a hyperbolic isometry and $K \subset \mathcal{H}(X)$ be a compact subset. Then for every exponent $s $ there is a constant $\beta = \beta(K,o,s) > 0$  so that
$$ e^{-s d_X(\gamma o, o)} \le \beta  e^{-\frac{s}{2} d_X(h o, o)} $$
for all elements $\gamma \in E_\Gamma(h, K)$. 
\end{lemma}
\begin{proof}
Consider some element $\gamma \in E_\Gamma(h, K)$ so that in particular $\gamma^{-1} g \gamma = h$ for some hyperbolic isometry  $g \in K$. The triangle inequality implies that
\begin{align*}
d_X(ho, o) &= d_X(\gamma^{-1} g \gamma  o, o) \le \\
&\le d_X(\gamma^{-1} o,o) + d_X( \gamma^{-1} g o, \gamma^{-1}  o) + d_X( \gamma^{-1}  g \gamma  o,  \gamma^{-1} g o) = \\
 &= 2 d_X(\gamma o,o) + d_X(go,o) 
\end{align*}
We obtain the following estimate
$$ e^{-s d_X(\gamma o, o)} \le   e^{\frac{s}{2} d_X(go,o)} e^{-\frac{s}{2} d_X(h o, o)} $$
and the lemma follows with the constant $\beta$ being given by 
$$ \beta = \sup_{g \in K} e^{\frac{s}{2}d_X(go, o)}. $$
\end{proof}

The following   property of uniform lattices is used  in Lemma \ref{lem:a bound for the shortest element in hyperbolic groups} below.

\begin{lemma}
\label{prop:uniform lattice has discrete orbits}
Let $\Gamma$ be a uniform lattice in $\mathrm{Isom}(X)$ and $\mathrm{C}$ a compact subset of $X$. Then there is a constant $n = n(\Gamma,C) > 0$ so that
$$ |\{\gamma \in \Gamma \: : \: \gamma g_1 C \cap g_2 C \neq \emptyset \} | \le n $$
for every pair of elements $g_1, g_2 \in \mathrm{Isom}(X)$.
\end{lemma}
\begin{proof}
Since $\Gamma$ is a uniform lattice there is a compact subset $ K \subset \mathrm{Isom}(X)$ satisfying $\Gamma K = \mathrm{Isom}(X)$. Denote $C' = KC$. In particular $C'$ is compact. Given any pair of elements $g_1, g_2 \in \mathrm{Isom}(X)$ there are $\gamma_1, \gamma_2 \in \Gamma$ so that $g_i \in \gamma_i K$ for $i \in \{1,2\}$. It suffices therefore to bound the size of the set
$$ \{\gamma \in \Gamma \: : \: \gamma \gamma_1 C' \cap \gamma_2 C' \neq \emptyset \}. $$
Such an  upper bound   exists as $\Gamma$ is acting properly on $X$.
\end{proof}

\begin{lemma}
\label{lem:a bound for the shortest element in hyperbolic groups}
Let $h\in \mathcal{H}(X)$ be a hyperbolic isometry and $K \subset \mathcal{H}(X)$   a compact subset. If  $E_\Gamma(h, K) \neq \emptyset$ then there exists an element  $\gamma_h \in E_\Gamma(h, K)$ 
such that
$$ \mathcal{P}_\Gamma(s;h,K) \le \alpha  e^{- s d_X(\gamma_h o,o) } $$  
for some constant $\alpha = \alpha(K,\Gamma,o,s) > 0$ depending on the exponent $s$ but independent of $h$.
%
\end{lemma}

\begin{proof}
It follows from Proposition \ref{prop:compact set of hyperbolic elements meets a ball}  that there is a sufficiently large radius $D = D(K, o)$ so that given any  hyperbolic isometry $g \in K$, every bi-infinite geodesic line $l : \mathbb{R} \to X$ with endpoints equal to $ \mathcal{E}(g) = (g^-, g^+) $ satisfies $l(\mathbb{R}) \cap \mathrm{B} \neq \emptyset$ where we denote $\mathrm{B} = B_X(D,o)$.  

Let $\mathrm{A}_h \subset X$ be any fixed bi-infinite  geodesic in $X$ with endpoints equal to $\mathcal{E}(h)$. 
Every element $\gamma \in E_\Gamma(h, K)$ satisfies $\gamma h \gamma^{-1} \in K$ by definition. Therefore for every such  $\gamma$ there exists  a point $x_\gamma \in \mathrm{A}_h$ with $ \gamma x_\gamma \in \mathrm{B} $. 
The triangle inequality implies  that
\begin{multline*} 
\mathcal{P}_\Gamma(s;h,K) =  \sum_{\gamma \in E_\Gamma(h, K)}e^{-sd_X(o, \gamma o ) } 
  \le \sum_{\gamma \in E_\Gamma(h, K)}e^{-s(d_X(\gamma o, \gamma x_\gamma) - d_X( \gamma x_\gamma, o))} \le \\
     \le  e^{sD}\sum_{\gamma \in E_\Gamma(h, K)}e^{-sd_X( o,  x_\gamma)} 
\end{multline*}
for every exponent $s > 0$.

Let  $p  \in \mathrm{A}_h$ denote  a fixed nearest point projection of the point $o$ to $\mathrm{A}_h$. The thin triangles condition implies the existence of some $\kappa > 0$, depending only on the hyperbolicity constant of $X$, such that given any other point  $x \in \mathrm{A}_h$ we have 
$$ d_X(x, o) \ge d_X(x, p) + d_X(p, o) - \kappa.  $$
This inequality and our previous estimate for the partial Poincare series give
$$
\mathcal{P}_\Gamma(s;h,K) \le e^{s(- d_X(p, o) + D+\kappa )}  \sum_{\gamma \in E_\Gamma(h, K)} e^{-sd_X( x_\gamma,  p)}
$$
for every exponent $s > 0$. 
The  series that appears  in the last equation is essentially a geometric progression.
 To see this, we denote  for every $ r \in \mathbb{N}$  
$$ E_r(h, K) = \{ \gamma \in E_\Gamma(h, K) \: : \: \text{ $\exists  x_\gamma \in \mathrm{A}_h$ with $d(p, x_\gamma) < r$ and $\gamma x_\gamma \in \mathrm{B}$} \}. $$
Let $r_0 \in \mathbb{N}$ be the minimal number so that $|E_{r_0}(h, K)| > 0$ and choose an arbitrary element $\gamma_h \in E_{r_o}(h,K)$. In particular
$$ E_\Gamma(h, K) = \bigcup_{r \ge r_0} E_r(h, K). $$
According to  Lemma \ref{prop:uniform lattice has discrete orbits}   there is a  number $n > 0$ depending on the subset $K$, the subgroup $\Gamma$ and the basepoint $o$  but independent of the element $h$ so that
$$ |E_r(h, K) \setminus E_{r-1}(h, K)| \le n  $$
for every $r \in \mathbb{N}$. 
We obtain 
$$ \sum_{\gamma \in E_\Gamma(h, K)} e^{-sd_X( x_\gamma,  p)} \le \sum_{r \ge r_0} \sum_{\gamma \in E_r(h, K) \setminus E_{r-1}(h, K)}  e^{-sr} < \alpha' e^{-sr_0} $$
where $ \alpha'$ is the constant $ \alpha'(K, \Gamma, o, s) = n (1-e^{-s})^{-1} $. 
By the triangle inequality and the choice of $\gamma_h$ we have 
%
\begin{multline*}
d_X(\gamma_h o,o) \le 
d_X(\gamma_h o, \gamma_h x_{\gamma_h} ) + d_X(\gamma_h x_{\gamma_h}, o) \le  d_X(o, x_{\gamma_h})  + D \le \\
\le d_X(o, p) +d_X(p, x_{\gamma_h}) + D \le d_X(o, p) + r_0 + D.
\end{multline*}
Putting everything together gives us
$$ 
\mathcal{P}_\Gamma(s;h,K) \le 
\alpha' e^{s(- d_X(p,o) -r_0 +D+\kappa  )}   \le 
\alpha e^{-sd_X(\gamma_h o,o)}
$$
where   $\alpha = \alpha(K,\Gamma,o,s)$ is the constant $ \alpha = \alpha'  e^{s(2D + \kappa)}$.
\end{proof}


\section{Quantitative recurrence for actions of hyperbolic groups}
\label{sec:quantitave recurrence}

Let $X$ be a proper geodesic Gromov hyperbolic metric space and  $G$ be a closed subgroup of $\mathrm{Isom}(X)$ acting cocompactly on $X$ and admitting a uniform lattice. Fix an arbitrary  basepoint $o \in X$.

\subsection*{Asymptotic notations}
\label{subs:asymptotic notation}
In what follows it will be convenient to introduce an  asymptotic notation, writing $a \asymp_c b$ for some constant $ c \ge 1$ to mean
$$ c^{-1}a \le b \le ca. $$
Clearly $a \asymp_{c_1} b$ and $b \asymp_{c_2} d$ together imply $a \asymp_{c_1c_2} d$.

Similarly, if $\mu$ and $\nu$ are a pair of Borel measures on the same Borel space we will write $\mu \asymp_c \nu$ for some constant $ c > 1$ to mean that $\mu$ and $\nu$ are absolutely continuous with respect to each other and their Radon--Nykodim derivative satisfies
$$ \frac{d\mu}{d\nu} \asymp_c 1.$$ 

\subsection*{Growth of orbits}
\label{sub:growth of orbits}
Let $\Gamma$ be a uniform lattice in $G$. Given  a pair of positive radii $0 \le r_1 \le r_2 $ we let $A_\Gamma\left[r_1,r_2\right]$ denote  the \emph{annulus}
$$A_{\Gamma}\left[r_1,r_2\right] = \{\gamma \in \Gamma \: : \: r_1 \le d_X(\gamma o, o) \le r_2\}.$$
It was shown by Coornaert that the critical exponent $\delta(\Gamma)$ controls the growth of annuli in the lattice $\Gamma$.

\begin{theorem}[{\cite[Theorem 7.2]{Coornaert}}]
\label{thm:Coornaert}
 For all sufficiently large $k > 0$ there is a constant $a  = a(k) > 0$ with
$$ |A_\Gamma\left[r,r + k\right]| \asymp_{a} e^{\delta(\Gamma)r}$$
for all $r\in \mathbb{N}$.
\end{theorem}
In fact Coornaert's theorem  holds for any convex-cocompact subgroup.

\subsection*{The maximal ergodic theorem and quantitative recurrence}

Fix a uniform lattice $\Gamma$ in the group $\mathrm{Isom}(X)$ as well as a Borel space $(Z,\mu)$  admitting an ergodic probability measure preserving action of $\Gamma$.

The following maximal ergodic theorem for probability measure preserving actions of lattices in Gromov hyperbolic spaces  was recently obtained by Bowen and Nevo  \cite[Theorem 1.1]{Bowen-Nevo-negcurved}.

\begin{theorem}[Maximal ergodic theorem for hyperbolic groups]
\label{thm:recurrenceassumption}
For all sufficiently large $ k > 0$ there is a constant $b = b(k) > 0$ 
such that for any    vector  $f\in L^{2}(X,\mu)$  we have that
$$ \| \sup_{r\in\mathbb{N}} e^{-\delta(\Gamma) r} \sum_{g\in A_\Gamma[r,r+k]} g_* f \|_2 \le b \|f\|_2. $$
\end{theorem}







We obtain the following straightforward quantitative recurrence result as a corollary of Theorem \ref{thm:recurrenceassumption}.
Recall that given a Borel subset $U \subset Z$ and a point $x \in Z$ we denote
$$ E_\Gamma(x,U) = \{ g\in \Gamma \: : \: gx \in U \}. $$
Moreover  for every $r, k > 0$ let
$$ E_{r,k}(x,U) = E_\Gamma(x,U) \cap A_\Gamma[r,r+k]. $$

\begin{cor}[Quantitive recurrence]
\label{cor:quantitive recurence}
For every sufficiently large $k > 0$ there exists a constant $0 < \kappa < 1$ so that  every Borel subset $U \subset Z$ with $\mu(U) > 1 - \kappa$ satisfies
$$ \mu\left(\{x \in Z \: : \: \inf_{r \in \mathbb{N}} e^{-\delta(\Gamma) r}  |E_{r,k}(x,U)| > \kappa \} \right) > \kappa. $$
\end{cor}
\begin{proof}
Let $k > 0$ be sufficiently large so that there are  constants $a = a(k) $ and  $b = b(k) $ as in Theorems \ref{thm:Coornaert} and \ref{thm:recurrenceassumption}. The maximal ergodic theorem implies that 
\begin{align*}
b \mu(U^c) &\ge \|\sup_{r\in\mathbb{N}} e^{-\delta(\Gamma)r} \sum_{g\in A_\Gamma[r,r+k]} g_* (\chi_X - \chi_{U}) \|_2   \ge \\
&\ge \|\sup_{r\in\mathbb{N}} (\frac{1}{a} \chi_X - e^{-\delta(\Gamma)r} \sum_{g\in A_\Gamma[r,r+k]} g_*  \chi_{U} ) \|_2 = \\
&= \|\frac{1}{a} \chi_X - \inf_{r\in\mathbb{N}} (  e^{-\delta(\Gamma)r} \sum_{g\in A_\Gamma[r,r+k]} g_*  \chi_{U} ) \|_2 \ge \\
&\ge \frac{1}{a} \|\chi_X\|_2 - \|\inf_{r\in\mathbb{N}} (  e^{-\delta(\Gamma)r} \sum_{g\in A_\Gamma[r,r+k]} g_*  \chi_{U} ) \|_2
\end{align*}
We obtain
$$ \|\inf_{r\in\mathbb{N}} e^{-\delta(\Gamma)r} \sum_{g\in A_\Gamma[r,r+k]} g_*  \chi_{U}  \|_2 \ge \frac{1}{a} - b \mu(U^c) > 0 $$
provided that $\mu(U) > 1- \kappa$ with $ 0 < \kappa < 1$ sufficiently small so that $ ab\kappa < 1 $.
\end{proof}

\begin{cor}
\label{cor:quantitave series from ergodic theorem}
There is a sufficiently small $\kappa>0$ such that every Borel subset $U \subset Z$ with $\mu(U) > 1- \kappa$ satisfies
$$ \mu\left( \{x \in Z \: : \: 
\mathcal{P}_\Gamma(\delta(\Gamma); x, U) = \infty \} \right) > \kappa $$
\end{cor}

\begin{proof}
Let $k > 0$ be sufficiently large and $\kappa>0$  sufficiently small as in Corollary \ref{cor:quantitive recurence}. 
In particular, given any Borel subset $U \subset Z$ with $\mu(U) > 1-\kappa$ there exists some  Borel subset $Y \subset Z$ with $\mu(Y) > \kappa$  such that  every point $x \in Y$ satisfies
$$ \inf_{r\in \mathbb{N}} e^{-\delta(\Gamma)r} |E_{r,k}(x,U)|> \kappa .$$
It follows that for every point $x \in Y$ 
\begin{align*}
\mathcal{P}_\Gamma(\delta(\Gamma); x, U) &= \sum_{\gamma \in E_\Gamma(x,U) }e^{-\delta(\Gamma) d_X(\gamma o,o)} \ge \\ &\ge \sum_{n \in \mathbb{N}} \sum_{\gamma \in E_{nk, k}(x,U)} e^{-\delta(\Gamma) d_X(\gamma o,o) } \ge \\
&\ge \sum_{n \in \mathbb{N}} e^{-\delta(\Gamma)(n+1)k} |E_{nk,k}(x,U)| \ge e^{-\delta(\Gamma)k} \sum_{n\in\mathbb{N}} \kappa = \infty
\end{align*}
as required.

 
 
\end{proof}



\section{Discrete invariant random subgroups}

\subsection*{Invariant random subgroups}
\label{sub:IRS}
Assume $X$ is a proper geodesic Gromov hyperbolic metric space. 
Let $G$ be a closed subgroup of $\mathrm{Isom}(X)$. We will use $\mathrm{Sub}(G)$ to denote the space of all closed subgroups of $G$ endowed with the Chabauty topology\footnote{This topology is generated by sub-basic Chabauty sets of the form $\{\text{$H\le G$ closed} \: : \: H \cap U \neq \emptyset\}$ for all open subsets $U \subset G$ as well as of the form $\{\text{$H\le G$ closed} \: : \: H \cap K = \emptyset\}$ for all compact subsets $K \subset G$.}. Recall that the space $\mathrm{Sub}(G)$ is compact. The group $G$ admits a continuous action on $\mathrm{Sub}(G)$ by conjugation. 

An \emph{invariant random subgroup} of $G$ is a conjugation invariant Borel probability measure on $\Sub{G}$. For example, every lattice $\Gamma$ in $G$ gives rise to an invariant random subgroup by pushing forward the $G$-invariant probability measure on $G/\Gamma$ to $\mathrm{Sub}(G)$ via the map $g\Gamma \mapsto g\Gamma g^{-1}$. More generally, given a normal subgroup $N$ in some lattice $\Gamma$ in $G$, we obtain an invariant random subgroup of $G$ by pushing forward  the Dirac mass at the  point  $N \in \mathrm{Sub}(\Gamma)$ in the same way.

We say that an invariant random subgroup is discrete, of convergence or divergence type, geometrically dense, etc, if this property is satisfied on a $\mu$-conull set of closed subgroups.
An invariant random subgroup $\mu$ is \emph{ergodic} if the $G$-action on $(\mathrm{Sub}, \mu)$ is ergodic.

In the current work we are mostly interested in discrete invariant random subgroups. This turns out to be  essentially always the case for semisimple linear groups over local fields. 









\begin{theorem}[\cite{7S, GL}]
\label{thm:borel density theorem for IRS}
Let $G$ be a simple linear group over a local field\footnote{In the positive characteristic case Theorem \ref{thm:borel density theorem for IRS} has an additional  very mild  technical  assumption.}.   If $\mu$ is a non-atomic invariant random subgroup of $G$ then $\mu$ is discrete and Zariski dense.
\end{theorem}

Any group $G$ in Theorem \ref{thm:borel density theorem for IRS} which has rank one  is a closed cocompact subgroup in the group of isometries of a certain proper geodesic Gromov hyperbolic space, namely the corresponding rank one symmetric space or Bruhat--Tits tree.


\subsection*{The critical exponent of an invariant random subgroup}

Let $\mathrm{DSub}(G)$ be the subspace of the Chabauty space $\mathrm{Sub}(G)$ consisting of discrete subgroups. 
The  \emph{critical exponent} $\delta(\mu)$ of the discrete invariant random subgroup $\mu$ is defined to be the expected value 
$$\delta(\mu) = \int_{\mathrm{DSub}(G)} \delta(\Gamma) \; \mathrm{d}\mu(\Gamma).$$ 
The critical exponent $\delta(\mu)$ is well defined and lies in the closed interval $\left[0,\dim_\textrm{H}(\partial X)\right]$ by Proposition \ref{prop:critical exponent is measurable, conjugation invariant} of the appendix. If $\mu$ is  an ergodic invariant random subgroup then  $\delta(\Gamma) = \delta(\mu)$ for $\mu$-almost every discrete subgroup $\Gamma$ of $G$.

\subsection*{Discrete  invariant random subgroups in Gromov hyperbolic spaces}
\label{sec:IRS}

Let $\Gamma$ be a uniform lattice in $G$.  Recall that $\mathcal{H}$ denotes the subset of $G$ consisting of all hyperbolic elements. We assume that $G$ is non-elementary so that in particular $\dim_\text{H}(\partial X) > 0$.

\begin{prop}
\label{prop:existence of a neighborhood having positive probability}
Let $\mu$ be an infinite invariant random subgroup of $G$. Then for every $ 0 < \kappa < 1$ there is an open relatively  compact  subset $V \subset \mathcal{H}$ such that $$\mu( \{ \text{$H \le G$ closed} \: : \: H \cap V \neq \emptyset \}) > 1-\kappa.$$
\end{prop}
\begin{proof}
We first claim that $\mu$-almost every closed subgroup contains hyperbolic elements. 
Since $\mu$-almost every closed subgroup $H$ is infinite we may write the invariant random subgroup $\mu$ as a convex combination 
$$ \mu = \theta \mu_\text{parabolic} + (1-\theta) \mu_\text{non-parabolic} $$ for some $0 \le \theta \le 1$. Here $\mu_\text{parabolic}$-almost every closed subgroup $H$ has $|\Lambda(H)| = 1$ and $\mu_\text{non-parabolic}$-almost every closed subgroup contains hyperbolic elements.
It suffices to show that  $\theta = 0$. 

Given a subgroup $H$ with $|\Lambda(H)| = 1$ let $\lambda_H \in \partial X$ denote the unique limit point of $H$.  We may push forward $\mu_\text{parabolic}$ via the map $H \mapsto \lambda_H$ to obtain a $G$-invariant probability measure on the Gromov boundary of $X$. This is impossible, provided that $G$ itself is non-elementary. Therefore $\theta = 0$ as required. Thus we have shown that $\mu$-almost every subgroup contains hyperbolic elements. 

Finally note that the collection $\mathcal{H}$ of hyperbolic elements in $G$ is open by  Proposition \ref{prop:hyperbolic elements in CAT minus one groups are open} and in particular can be exhausted by a countable collection of relatively compact open subsets. This completes the proof.
\end{proof}


We are now ready to present the proof of one of our main results. 




\begin{thm}
\label{thm:divergence at one half of critical exponent}
Let  $\mu$ be a discrete infinite invariant random subgroup of $G$.  Assume that $G$ is non-elementary. Then
$$ \mathcal{P}_\Delta\left(\frac{\dim_\mathrm{H}(\partial X)}{2}\right) = \infty $$
for $\mu$-almost every closed subgroup $\Delta$ of $G$.
\end{thm}
This implies of course that $\delta(\mu) \ge \frac{1}{2} \dim_\mathrm{H}(\partial X) > 0$. In particular, since $\mu$-almost every subgroup contains hyperbolic elements  $\mu$ must be non-elementary. 
The strict inequality $\delta(\mu) >\frac{1}{2} \dim_\mathrm{H}(\partial X)$ as in the statement of Theorem \ref{thm:main theorem} will be established in \S\ref{sec:invariant random subgroups of divergence type} below using Patteron--Sullivan theory. 
\begin{proof}
The ergodic decomposition theorem allows us to write the invariant random subgroup $\mu$ as a convex combination of ergodic invariant random subgroups. We may therefore assume without loss of generality that $\mu$ is ergodic.


 Fix a sufficiently small constant $ 0 < \kappa < 1$ so that Corollaries \ref{cor:quantitive recurence} and \ref{cor:quantitave series from ergodic theorem} hold true with respect to the probability measure preserving action of the lattice $\Gamma$ on  the Borel space $(\mathrm{Sub}(G),\mu)$. Let $V \subset \mathcal{H} \subset G$ be an open relatively compact  subset of hyperbolic elements as provided by Proposition \ref{prop:existence of a neighborhood having positive probability}, so that $\mu(\Omega_V) > 1-\kappa$ where
$$ \Omega_V = \{ \text{$H \le G$ closed} \: : \: H \cap V \neq \emptyset \} \subset \mathrm{Sub}(G).$$

To simplify notation let $\delta = \dim_\mathrm{H}(\partial X)$.
Observe that every  element $h \in G$ with $E_\Gamma(h,V) \neq \emptyset$ is necessarily hyperbolic. According to Lemma \ref{lem:a bound for the shortest element in hyperbolic groups} every such element $h$ satisfies
$$   e^{- \delta  d_X(\gamma_h o,o) } \ge \frac{1}{\alpha} \mathcal{P}_\Gamma(\delta; h, V)   $$
for some particular  choice of $\gamma_h \in E_\Gamma(h,V)$, where $\alpha > 0$ is a constant independent of $h$.
We obtain the following estimate for the Poincare series $\mathcal{P}_\Delta(\delta/2)$ of any discrete subgroup $\Delta$ of $G$
$$ \mathcal{P}_\Delta(\delta/ 2 ) = \sum_{h \in \Delta} e^{-\frac{\delta}{2} d_X(ho,o) } \ge \frac{1}{ \beta} \sum_{\mathclap{\substack{h\in\Delta \\ E_\Gamma(h,V) \neq \emptyset}}} e^{- \delta  d_X(\gamma_h o,o) } \ge \frac{1}{\alpha \beta} \sum_{h\in\Delta} \mathcal{P}_\Gamma(\delta; h, V).   $$
In the above estimate we used Lemma  \ref{lemma:introducing one half}  to compare  $d_X(ho, o)$ and $d_X(\gamma_h o, o)$, with the constant   $\beta > 0$ being as in that lemma.

The validity of the condition  $\gamma \Delta \gamma^{-1} \in \Omega_V$ for some element $\gamma \in \Gamma$ implies  by definition  that $\gamma h \gamma^{-1} \in V$ for some $h \in \Delta$, and therefore
$$\sum_{h\in\Delta} \mathcal{P}_\Gamma(\delta; h, V) \ge  \mathcal{P}_\Gamma(\delta;\Delta, \Omega_V). $$ 
According to Corollary \ref{cor:quantitave series from ergodic theorem} there is a Borel subset $Y \subset \mathrm{DSub}(G)$ with $\mu(Z) > \kappa$ and such that $\mathcal{P}_\Gamma(\delta;\Delta, \Omega_V)   = \infty  $ for every subgroup $\Delta \in Y$. Therefore  $\mathcal{P}_\Delta(\delta/ 2 ) = \infty$ holds true for every $\Delta \in Y$ and hence $\mu$-almost surely by the ergodicity of $\mu$.
\end{proof}



\section{Quasiconformal densities}

Let $X$ be a proper geodesic Gromov hyperbolic metric space.



\subsection*{Busemann functions and shadows}

We recall some basic definitions and facts.

\begin{definition}
\label{def:shadow}
 Given any two points $x,y \in X$ the \emph{shadow}  from $x$ of the ball of radius $R > 0$  around $y$  is
 $$\mathcal{S}_R(x,y) = \{ \xi \in \partial X \: : \: \text{any geodesic ray from $x$ to $\xi$ intersects $B_R(y)$} \}. $$
\end{definition}

Fix a uniform lattice $\Gamma$ in $\mathrm{Isom}(X)$ and a basepoint $ o \in X$.

\begin{prop}
\label{prop:covering lemma for shadows}
For all sufficiently large  radii $k, R > 0$ there is an integer $p  \in \mathbb{N}$ so that for every $ r > 0$
$$   \bigcup_{\gamma \in A_\Gamma[r, r+k]} \mathcal{S}_R(o,\gamma o) = \partial X$$   
and moreover each ideal point $\xi \in \partial X$ belongs to at most $p$ of these shadows.
\end{prop}

\begin{proof}
Let $D > 0$ be a sufficiently large radius so that $\Gamma B_X(o,D) = X$. Consider an ideal point  $\xi \in \partial X$ and let $l : \mathbb{R}_{\ge 0} \to X$ be any  geodesic ray with $l(0) = o$ and $l(\infty) = \xi$.  There exists some element $\gamma \in \Gamma$ so that $d_{X}(\gamma o,l(r + D))\leq D$. Therefore $l(\mathbb{R}_{\ge 0})$ intersects $B_{X}(\gamma o,2D)$. Recall that any other geodesic ray from $o$ to $\xi$ lies in the $\delta'$ neighborhood of $l$ where $\delta'$  is a constant. Taking $k  = 2D$ and $R = 2D + \delta'$, we find that any ideal point  $\xi \in \partial  X$ lies in some shadow $S_{R}(o,\gamma o)$ with $ \gamma \in A_\Gamma[r,r+k]$ as required. Moreover, observe that any other shadow $S_R(o,\gamma' o)$ with $\gamma' \in A_\Gamma[r,r+k]$ and containing $\xi$ must satisfy $d_X(\gamma' o, l(r+D)) < R + D$. 
As $\Gamma$ is a uniform lattice, the existence of an upper bound $p \in \mathbb{N}$ on the number of such shadows containing a given ideal point follows from Lemma \ref{prop:uniform lattice has discrete orbits}.
%
%
\end{proof}

\begin{prop}[{\cite[Lemma 6.3]{Coornaert}}]
\label{prop:relation between shadows and visual metric}
Assume that $\mathrm{Isom}(X)$ admits a uniform lattice. Let $\rho_o$ be a visual metric on $\partial X$ with respect to the basepoint $o \in X$. Given a point  $x \in X$ distinct from $o$ let $\xi_x \in  \partial X$ be the endpoint of any geodesic ray starting at $o$ and passing through $x$. Then 
$$ \sup \{ \rho_o(\zeta, \xi_x) \: : \: \text{$x$ distinct from $o$ and $\zeta \in \partial X \setminus S_R(x,o)$} \} $$ 
tends to $0$ as $R \to \infty$.
\end{prop}
\begin{definition}
The \emph{Busemann function} $\beta_\xi : X \times X \to \mathbb{R}$ associated to the ideal point $\xi \in \partial X$ is given by
$$\beta_{\xi}(x,y)=\lim \sup_{z\to \xi} \left(d_{X}(x,z)-d_{X}(y,z)\right)$$
(here  $z \in X$ is a point tending to $\xi$ along any geodesic ray with endpoint $\xi$).
\end{definition}

\begin{lemma}[{\cite[Lemma 6.2]{Coornaert}}]
\label{prop:comparison between shadow and buseman}
Let $x,y \in X$ be a pair of distinct points and $R > 0$ a radius. Then for every $\xi \in \mathcal{S}_R(x,y)$ we have that
$$ d_X(x,y) - 2R \le \beta_\xi(x,y) \le d_X(x,y). $$
\end{lemma}


\subsection*{Quasiconformal densities}

Consider a fixed non-elementary discrete  subgroup $\Gamma$ of $\mathrm{Isom}(X)$. We use below the asymptotic notation introduced on p. \pageref{subs:asymptotic notation} of  \S\ref{sec:quantitave recurrence}. We use $\|\mu\|$ to denote the total mass a finite positive Borel measure.

\begin{definition}
\label{def:conformal density}
A  $\Gamma$-\emph{quasiconformal density} of dimension $s \ge 0$ and distortion $d \ge 1$  is a family $\{\nu_{x}\}_{x\in X}$ of  finite Borel measures on $\partial X$ such that
\begin{enumerate}
\item the support of $\nu_x$ is  the limit set $\Lambda(\Gamma) $ for all points $x \in X$, 
\item
$\frac{d\nu_{x}}{d\nu_{y}}(\xi)\asymp_{d} e^{s\beta_{\xi}(y,x)}$
for $\nu_{y}$-almost every  point  $\xi \in \partial X$, and

\item
$g_{*}\nu_{x} \asymp_{d} \nu_{gx}$ for all elements $g\in \Gamma$ and points $x \in X$. 
\end{enumerate}
\end{definition}

Recall that $(g_{*}\nu)(A)=\nu(g^{-1}A)$. As long as   $\dim_\text{H}(\partial X) <\infty$, any discrete subgroup $\Gamma$  admits a quasiconformal density of dimension $\delta(\Gamma)$.
If $\Gamma$ is of divergence type then a $\Gamma$-quasiconformal density  of dimension $\delta(\Gamma)$ is ergodic and essentially unique up to a bounded multiplicative factor, in the following sense.

\begin{prop}\label{prop:uniquessness of conformal density} 
Assume that $\Gamma$ has divergence type and let   $\{\nu_{x}\}_{x\in X}$ and $\{\eta_{x}\}_{x\in X}$ be two $\Gamma$-quasiconformal densities of the same dimension $s \ge 0$ and distortion $d \ge 1$.
Then for all points $x \in  X$ 
$$ \frac{\|\nu_o\|}{\|\eta_o\|} \asymp_{d^4} \frac{\|\nu_x\|}{\|\eta_x\|} $$
\end{prop}
\begin{proof}
According to \cite[Theorem 5.2]{MYJ} we have the estimate
$$ \frac{\nu_o(\xi)}{\eta_o(\xi)} \asymp_{d^2} \frac{\|\nu_o\|}{\|\eta_o\|} $$
for all boundary points $\xi \in \partial X$. Combined with property (2) given in Definition \ref{def:conformal density} of quasiconformal densities this implies
$$ \frac{\nu_o(\xi)}{\nu_x(\xi)} \asymp_{d^2} \frac{\eta_o(\xi)}{\eta_x(\xi)}  \Rightarrow  \frac{\nu_o(\xi)}{\eta_o(\xi)} \asymp_{d^2} \frac{\nu_x(\xi)}{\eta_x(\xi)}  \Rightarrow  \frac{\|\nu_o\|}{\|\eta_o\|} \asymp_{d^4 } \frac{\nu_x(\xi)}{\eta_x(\xi)}.$$
The result follows by integrating the above estimate over all points $\xi$ in $\partial X$.
\end{proof}

The following proposition allows us to relate  quasiconformal densities of conjugate subgroups in a straightforward manner.
\begin{prop}
\label{prop:on conjugating a conformal density}
Fix an element $g \in \mathrm{Isom}(X)$. If $\{\nu_{x}\}_{x\in X}$ is a $\Gamma$-quasiconformal density then $\{g_{*}\nu_{g^{-1}x}\}_{x\in X}$ is a $(g\Gamma g^{-1})$-quasiconformal density of the same dimension and distortion.
\end{prop}

\begin{proof}
Let $\{\nu_x\}_{x\in X}$ be a $\Gamma$-quasiconformal density of dimension $s \ge 0$ and distortion $d \ge 1$.  Note that
$$\frac{dg_*  \nu_{g^{-1}x}}{ d g_* \nu_{g^{-1} y}}(\xi)=\frac{d\nu_{g^{-1}x}}{ d\nu_{g^{-1} y}}(g^{-1}\xi)\asymp_d e^{s\beta_{g^{-1}\xi}(g^{-1}  y,g^{-1}  x)} = e^{s\beta_{\xi}(y,x)}$$
Let $g\gamma g^{-1}$  be an arbitrary element of $g\Gamma g^{-1}$. Therefore
$$ (g \gamma g^{-1})_* g_*\nu_{g^{-1}x} = (g \gamma)_* \nu_{g^{-1}x} \asymp_d g_* \nu_{\gamma g^{-1} x} = g_* \nu_{g^{-1} (g \gamma g^{-1}) x}$$
It follows that $\{g_* \nu_{g^{-1} x}\}_{x \in X}$ is a $(g\Gamma g^{-1})$-quasiconformal density of the same dimension and distortion, as required.

\end{proof}

\begin{remark}
\label{rem:quasi becoming strict 1}
In the special case that the distortion is equal to one, a quasiconformal density is called  a \emph{conformal density}. In that case all of the above statements hold in a strict sense, and not just up to a multiplicative constant. This happens for example if $X$ is a $\mathrm{CAT}(-1)$-space. However this assumption  does not simplify the proof in any essential way. 
\end{remark}





\section{Invariant random subgroups of divergence type}
\label{sec:invariant random subgroups of divergence type}

Let $\mu$ be a discrete ergodic invariant random \emph{of divergence type}.

\subsection*{The Poincare quasi-cocycle}

There exists a \emph{$\mu$-measurable family $\nu$ of quasiconformal densities} of dimension $\delta(\mu)$. This means that $\nu$ is a $\mu$-measurable map
$$ \nu : \mathrm{DSub}(G) \times X  \to \mathcal{M}(\partial X), \quad \nu : (\Gamma,x)  \mapsto \nu_x^\Gamma $$ 
so that  the family $\{\nu^\Gamma_x\}_{x\in X}$ is a $\Gamma$-quasiconformal density of dimension $\delta(\mu)$ and some fixed distortion $d \ge 1$ for $\mu$-almost every discrete subgroup $\Gamma$. We  will additionally assume without loss of generality that the family $\nu$ is normalized so that  $$\|\nu^\Gamma_o\| = \nu^\Gamma_o(\partial X) = 1.$$ 
Since $\mu$ is of divergence type the existence of a $\mu$-measurable quasiconformal density is guaranteed by Proposition \ref{prop:existance of a measurable quasiconformal density} of the appendix.


Associated to the $\mu$-measurable quasiconformal density is a \emph{Poincare quasi-cocycle} $\pi_\nu$. This is simply the map 
$$\pi_\nu : G \times \mathrm{DSub}(G) \to \mathbb{R}_{>0}, \quad \pi_\nu(g,\Gamma)= \|\nu^{\Gamma}_{g^{-1}o} \| = \nu^{\Gamma}_{g^{-1}o}(\partial X). $$
Since $\mu$ is of divergence type we have by Propositions \ref{prop:uniquessness of conformal density} and \ref{prop:on conjugating a conformal density} that
 $$ g_{*}\nu^\Gamma_{g^{-1}x} \asymp_{d^{2}}  \frac{\| g_* \nu_{g^{-1}x}^\Gamma \| }{\| \nu^{g\Gamma g^{-1}}_{x}\| } \nu^{g\Gamma g^{-1}}_{x} \asymp_{d^4} \frac{\| g_* \nu_{g^{-1}o}^\Gamma \| }{\| \nu^{g\Gamma g^{-1}}_{o}\| } \nu^{g\Gamma g^{-1}}_{x} =  \|  \nu_{g^{-1}o}^\Gamma \|  \nu^{g\Gamma g^{-1}}_{x}$$
 for every element $g \in G$ and all points $x \in X$. In other words, the Poincare quasi-cocycle satisfies the relation
$$ g_{*}\nu^\Gamma_{g^{-1}x} \asymp_{d^{6}} \pi_\nu(g,\Gamma) \nu^{g\Gamma g^{-1}}_{x}. $$
This implies that a  cocycle relation holds for $\pi_\nu$ in an approximate sense, namely 
$$ \pi_\nu(gh,\Gamma) \asymp_{d^{12}}  \pi_\nu(g,h\Gamma h^{-1}) \pi_\nu(h,\Gamma) $$
and in particular
$$ \pi_\nu(g^{-1}, g \Gamma g^{-1}) \pi_\nu(g,\Gamma) \asymp_{d^{12}} 1$$
for every pair of elements $g, h \in G$ and $\mu$-almost every closed subgroup $\Gamma$.





\begin{remark}
	\label{rem:quasi becoming strict 2}
If every discrete subgroup of divergence type in $\mathrm{Isom}(X) $ admits a conformal (rather than quasiconformal) density then $\pi_v$ is a multiplicative cocycle in the usual sense. Whenever  $G$ has Kazhdan's property $(T)$ any such cocycle is trivial \cite[Theorem 9.1.1]{Zimmer} and  many of our arguments can be somewhat simplified.
\end{remark}

\subsection*{Sullivan's shadow lemma for invariant random subgroups}


Consider a cocompact subgroup $G$ of $\mathrm{Isom}(X)$ and assume that $\mu$ is an ergodic discrete non-elementary invariant random subgroup of divergence type in $G$. Fix a $\mu$-measureable quasiconformal density $\nu$. In addition fix a basepoint $o \in X$.

\begin{prop}[Shadow lemma]
\label{prop:shadow lemma for IRS}
Fix an arbitrary small constant $0 < \kappa < 1$. There is a Borel subset $Y \subset \mathrm{DSub}(G)$ with $\mu(Y) > 1- \kappa $
 such that for every sufficiently large radius $R > 0$ there is a constant $c = c(R)> 0$ with
 $$  \nu_o^\Gamma( \mathcal{S}_R(o, g^{-1} o)) \asymp_c \pi_\nu(g,\Gamma)  e^{-\delta(\mu)d_X( o, go)} $$
for every $g \in G$ 
and $\mu$-almost every closed subgroup $\Gamma$ with $g \Gamma g^{-1} \in Y$.
\end{prop}


\begin{proof}
Recall that $\nu_o^\Gamma$ is normalized to be a probability measure and has full support on the limit set $\Lambda(\Gamma)$ for $\mu$-almost every subgroup $\Gamma$. 
 Since $\mu$ is non-elementary the limit set $\Lambda(\Gamma)$ is $\mu$-almost surely not  a single point and so
 $$\nu_o^\Gamma({\xi}) < \nu_o(\partial X) = 1$$
for every ideal point $\xi \in \partial X$ and $\mu$-almost every subgroup $\Gamma$. 

Let $\rho_o$ be a visual metric on the Gromov boundary $\partial X$ with respect to the basepoint $o\in X$. The compactness of $\partial X$ implies that   there is a sufficiently small constant $\theta(\Gamma) > 0$ depending on the subgroup $\Gamma$ such that
$$ \nu_o^\Gamma\left(\{\zeta \in\partial X \: : \: \rho_o(\zeta,\xi) > \theta(\Gamma) \} \right)   > \theta(\Gamma) $$ 
holds for every ideal point $\xi \in \partial X$. Since the family $\nu$ of quasiconformal densities is $\mu$-measurable the function $\theta$ can be chosen to be $\mu$-measurable on $\mathrm{DSub}(G)$.
Moreover by Proposition \ref{prop:relation between shadows and visual metric}
there is a constant $r(\Gamma) > 0$  depending only on  $\theta(\Gamma)$ so that for each point $x\in X$  there exists  an ideal point $\xi \in \partial X$ with 
$$ \{\zeta \in \partial X\: : \: \rho_o(\zeta,\xi) > \theta(\Gamma) \} \subset \mathcal{S}_r(x,o) $$
for every radius $r > r(\Gamma)$. We may assume the function $r$ to be $\mu$-measurable.


Fix a pair of constants $\Theta > 0$ sufficiently small and $R > 0$ sufficiently large so that the following Borel set
$$ Y = \{\Gamma \in \mathrm{DSub}(G) \: : \: \text{$\theta(\Gamma) > \Theta, r(\Gamma) < R$} \} $$
satisfies $\mu(Y) > 1- \kappa$. 

Fix an element $g \in G$.  Consider a discrete subgoup $\Gamma$  such that its conjugate $ g \Gamma g^{-1}$ belongs to $Y$. In particular 
$$  \nu^{g \Gamma g^{-1}}_o(\mathcal{S}_R(go,o)) > \Theta. $$
By the definition of the Poincare quasi-cocycle $\pi_\nu$ we have that 
$$
\nu_o^\Gamma(\mathcal{S}_R(o,g^{-1}o)) = \nu_o^\Gamma(g^{-1}\mathcal{S}_R(go,o)) \asymp_{d^6}  \pi_\nu(g,\Gamma) \nu_{g o}^{g \Gamma g^{-1}}(\mathcal{S}_R(go,o)). $$
where $d \ge 1$ is the distortion of the family $\nu$. To complete the proof of the shadow lemma we are required  to estimate 
\begin{align*}
\nu_{g o}^{g \Gamma g^{-1}}(\mathcal{S}_R(go,o)) &= \int_{\mathcal{S}_R(go,o)} \mathrm{d}\nu_{go}^{g \Gamma g ^{-1}}(\xi) = \\
&= \int_{\mathcal{S}_R(go,o)} \frac{\nu_{go}^{g \Gamma g ^{-1}}(\xi)}{\nu_{o}^{g \Gamma g ^{-1}}(\xi)}\mathrm{d}\nu_{o}^{g \Gamma g ^{-1}}(\xi) \asymp_d  \\
&\asymp_d \int_{\mathcal{S}_R(go,o)} e^{- \delta(\nu)\beta_\xi(go,o)}\mathrm{d}\nu_{o}^{g \Gamma g^{-1}}(\xi)
\end{align*}
 where we have used the properties of quasiconformal densities in the final equation, see  Definition \ref{def:conformal density}.
Finally, for every $\xi \in \mathcal{S}_R(go,o)$ we have that
$$d_X(go,o) - 2R \le \beta_\xi(go,o) \le d_X(go,o)$$ by Lemma \ref{prop:comparison between shadow and buseman}. This concludes the proof with the constant 
$$ c = d^7 \Theta^{-1}  e^{2\delta(\mu)R} .$$ 
\end{proof}

\subsection*{Recurrence and the Poincare quasi-cocycle}
Let $\Gamma$ be a uniform lattice in $G$. For every discrete subgroup $\Delta \in \mathrm{DSub}(G)$, Borel subset $Y \subset \mathrm{DSub}(G)$ and radius $r > 0$  we let
$$ E_{r,k}(\Delta, Y) = E_\Gamma(\Delta, Y) \cap A_\Gamma[r,r+k]$$
where $A_\Gamma[r,r+k]$ is an annulus in the group $\Gamma$ with respect to its action on $X$.

\begin{prop}
\label{prop:summing over a cocycle}
 Let  $Y \subset \mathrm{DSub}(G)$ be the Borel subset with respect to which the shadow lemma (Proposition \ref{prop:shadow lemma for IRS}) holds.  Then for all sufficiently large $k > 0$ there is  a  constant $c' > 0$  such that
$$  \sum_{\gamma \in E_{r,k}(\Delta, Y) } \pi_\nu(\gamma,\Delta) \le c' e^{\delta(\mu) r}$$
for every $r >0$ and  $\mu$-almost every $\Delta \in \mathrm{DSub}(G)$.
\end{prop}
\begin{proof}

According to Proposition \ref{prop:covering lemma for shadows} for all sufficiently large radii $k, R > 0$ there is  an integer $p$ so that 
$$  \bigcup_{\gamma \in A_\Gamma\left[r,r+k \right]} \mathcal{S}_R(o,\gamma o) =  \partial X$$
for every $r > 0$ and moreover every ideal point $\xi \in \partial X$ belongs to at most $p$ distinct shadows. We may  assume that   $R$ is sufficiently large for us to be able to use the Sullivan shadow lemma for invariant random subgroups. Let  $c > 0$ be the corresponding constant, as in  Proposition \ref{prop:shadow lemma for IRS}.  

Denote $\Gamma_{r,k} = A_\Gamma[r,r+k]$.  By the shadow lemma for invariant random subgroups
\begin{align*}
\frac{1}{c} e^{- \delta(\mu)(r + k)}\sum_{\gamma \in E_{r,k}(\Delta,Y) } \pi_\nu(\gamma,\Delta) &\le  
\frac{1}{c} \sum_{\gamma \in E_{r,k}(\Delta,Y) } \pi_\nu(\gamma,\Delta) e^{-\delta(\mu) d_X(o,\gamma o)} \le \\
&\le  \sum_{\gamma \in E_{r,k}(\Delta,Y) } \nu_o^\Delta\left(\mathcal{S}_R(o,\gamma^{-1} o) \right) = (*)
\end{align*}
for every radius $r > 0$  and $\mu$-almost every closed subgroup $\Delta$.  The estimate for $(*)$ can only go up when summing over all of $\Gamma_{r,k}$ rather than just $E_{r,k}(\Delta, Y)$, so that
$$
(*) \le   \sum_{\gamma \in \Gamma_{r,k}} \nu_o^\Delta\left(\mathcal{S}_R(o,\gamma^{-1} o) \right) \le 
    p\nu_o^\Delta(\bigcup_{\gamma \in \Gamma_{r,k}}\mathcal{S}_R(o,\gamma^{-1} o) )  = p \nu^\Delta_o(\partial X) = p.
$$
The result follows with  the constant $c' = cpe^{\delta(\mu)k}$.
\end{proof}

\begin{prop}
\label{prop:cocycle and inverse trick}
 Let  $Y \subset \mathrm{DSub}(G)$ be any Borel subset. Then  
 $$  \int_Y \sum_{\gamma \in E_{r,k}(\Delta, Y)} \pi_\nu(\gamma,\Delta)  \, \mathrm{d}\mu(\Delta) \ge \frac{1}{2d^{12}} \int_Y |E_{r,k}(\Delta,Y)| \, \mathrm{d}\mu(\Delta)  $$
 for every $r, k > 0$, where $d$ is the distortion of the $\mu$-measurable family $\nu$. 
\end{prop}
\begin{proof}
Denote  $\Gamma_{r,k} = A_\Gamma[r,r+k]$. Given an element $\gamma \in \Gamma_{r,k}$ and a subgroup $\Delta \in Y$ it is  clear that 
$ \gamma \in E_{r,k}(\Delta,Y)$ if and only if $\Delta \in Y^{\gamma^{-1}} \cap Y$. 
This observation makes it possible to rewrite the   expression on the left hand side of the inequality  in the statement of the proposition, separating it into one summation over all unordered pairs of elements $\{\gamma,\gamma^{-1}\}$ with $\gamma^2 \neq e$  and another summation over all elements $\gamma$ of order two, and finally switching the integral and the sum using Fubini's theorem.   This gives
\begin{multline*}
\sum_{\{\gamma,\gamma^{-1}\} \subset \Gamma_{r,k}, \gamma^2 \neq e} \left(
\int_{Y^{\gamma^{-1}} \cap Y} \pi_\nu(\gamma,\Delta)\; \text{d}\mu(\Delta) + 
\int_{Y^{\gamma} \cap Y} \pi_\nu(\gamma^{-1},\Delta)\; \text{d}\mu(\Delta) 
\right) = \\
+ \sum_{\gamma \in \Gamma_{r,k}, \gamma^2 = e } \int_{Y^{\gamma } \cap Y} \pi_\nu(\gamma,\Delta) \; \text{d}\mu(\Delta) = (*)
\end{multline*}
Expression $(*)$  can be rearranged again by summing over \emph{all} unordered subsets $\{\gamma, \gamma^{-1}\} \subset \Gamma_r$. Since every element of order two is accounted for twice we obtain
$$
(*) \ge \frac{1}{2} \sum_{\{\gamma,\gamma^{-1}\} \subset \Gamma_{r,k}} 
\int_{Y^{\gamma^{-1}} \cap Y} \left( \pi_\nu(\gamma,\Delta) + \pi_\nu(\gamma^{-1},\Delta^\gamma) \right) \; \text{d}\mu(\Delta) = (**)
$$
Recall that $\pi_\nu$ is a multiplicative quasi-cocycle. In particular
$$ \pi_\nu(\gamma^{-1}, \gamma \Delta \gamma^{-1}) \pi_\nu(\gamma,\Delta) \asymp_{d^{12}} 1$$
which immediately implies that
$$\pi_\nu(\gamma,\Delta) + \pi_\nu(\gamma^{-1},\Delta^\gamma) \ge 2d^{-12}.$$
for every $\gamma \in \Gamma$ and $\mu$-almost every subgroup $\Delta$. The constant $d$ is the distortion of the $\mu$-measurable family $\nu$ of quasi-conformal densities.
 Proceeding with the estimate  $(**)$  we now obtain
\begin{align*}
(**)  &\ge  \frac{1}{2} \sum_{\{\gamma,\gamma^{-1}\} \subset \Gamma_{r,k}} 
\int_{Y^{\gamma^{-1}} \cap Y} \frac{2}{d^{12}}\; \text{d}\mu(\Delta) =\frac{1}{d^{12}} \sum_{\{\gamma,\gamma^{-1}\}\subset\Gamma_{r,k}} \mu(Y \cap Y^\gamma)  \ge \\
&\ge \frac{1}{2d^{12}} \sum_{\gamma \in \Gamma_{r,k}} \mu(Y \cap Y^\gamma) = \frac{1}{2d^{12}} \int_Y |E_r(\Delta,Y)| \; \mathrm{d}\mu(\Delta)
\end{align*}
as required.

\end{proof}

\subsection*{On the critical exponent of a divergence type  invariant random subgroup  }

\begin{theorem}
\label{thm:divergence type implies maximal critical exponent}
Let $\mu$ be  an  discrete invariant random subgroup of divergence type in $G$. Then  $\delta(\mu) = \dim_\mathrm{H}(\partial X)$.
\end{theorem}
\begin{proof}
Recall that the critical exponent  of any discrete subgroup of $\mathrm{Isom}(X)$ is bounded above by $\dim_\text{H}(\partial X)$. In particular $\delta(\mu) \le \dim_\text{H}(\partial X)$ is true in general. We  need to establish the reverse inequality, namely show that $\delta(\mu) \ge \dim_\text{H}(\partial X)$. 

The ergodic decomposition allows us to assume from now on and without loss of generality that $\mu$ is ergodic. Since $\mu$ is of divergence type it must be infinite. Therefore Proposition \ref{prop:existence of a neighborhood having positive probability} and Theorem \ref{thm:divergence at one half of critical exponent} imply that $\mu$ is non-elementary.

We begin be setting up the required objects. Fix a basepoint $o \in X$ and let $\Gamma$ be a uniform lattice in $G$. Let $\nu$ be a $\mu$-measurable family of quasiconformal densities of dimension $\delta(\mu)$ and distortion $d\ge 1$, normalized so that $\|\nu^\Delta_o\| = 1$. Let  $\pi_\nu$ be the associated Poincare quasi-cocycle.  Fix a sufficiently large $ k > 0$ and a sufficiently small constant  $0 < \kappa <1 $ so that  Corollary \ref{cor:quantitive recurence} of  quantitative recurrence holds. Finally,  let $Y \subset \mathrm{DSub}(G)$ be a Borel subset as in Proposition \ref{prop:shadow lemma for IRS} (i.e. the shadow lemma for invariant random subgroups) and satisfying $\mu(Y) > 1-\kappa$.

The quantitative recurrence property established in  Corollary \ref{cor:quantitive recurence} provides us with a Borel subset $U \subset Y$ with $\mu(U) > 0$ so that
$$   |E_{r,k}(\Delta,U)| \ge \kappa e^{ \delta(\Gamma) r}$$
for all $ r > 0$ and for  every $\Delta \in U$. Integrating this over $U$ gives
$$ \int_U  |E_{r,k}(\Delta,U)|  \; \mathrm{d}\mu(\Delta) \ge (\kappa \mu(U)) e^{\delta(\Gamma)r }  $$
for all $ r > 0$. On the other hand, by combining the two previous Propositions \ref{prop:summing over a cocycle} and \ref{prop:cocycle and inverse trick} we deduce that
$$ \int_Y |E_{r,k}(\Delta,Y)| \, \mathrm{d}\mu(\Delta) \le (2d^{12} c'\mu(Y)) e^{\delta(\mu) r} $$
for all $ r > 0$. Since $U \subset Y$ these two estimates can only be compatible for all sufficiently large $ r$ provided that $\delta(\mu) \ge \delta(\Gamma) = \dim_\textrm{H}(\partial X)$.
\end{proof}

\emph{Our proofs of Proposition \ref{prop:shadow lemma for IRS} and 
Theorem \ref{thm:divergence type implies maximal critical exponent} are inspired by the elegant expository paper  by Quint  \cite{Quint}, especially his Lemma 4.10 and Theorem 4.11.}

%



\section{Proofs of the main results}

We now complete the proofs of all the results stated in the introduction.

\begin{proof}[Proof of Theorem \ref{thm:main theorem}]
Let $G$ be a non-elementary group acting cocompactly on $X$ and  $\mu$  a discrete and infinite invariant random subgroup of $G$. The second part of Theorem \ref{thm:main theorem} follows immediately from Theorem \ref{thm:divergence type implies maximal critical exponent}. Moreover 
$$ \mathcal{P}_\Delta\left(\frac{\dim_\mathrm{H}(\partial X)}{2}\right) = \infty $$
for $\mu$-almost every closed subgroup $\Delta$ of $G$ according to Theorem \ref{thm:divergence at one half of critical exponent}.
In particular 
$\delta(\Delta) \ge \frac{1}{2}\dim_\mathrm{H}(\partial X) > 0$.
We claim that in fact $\delta(\Delta) > \frac{1}{2}\dim_\mathrm{H}(\partial X) $ happens $\mu$-almost always. For if this is not the case then $\delta(\Delta) = \frac{1}{2}\dim_\mathrm{H}(\partial X)$ and in particular $\Delta$ is of divergence type $\mu$-almost surely.  Therefore   $\delta(\Delta) = \dim_\mathrm{H}(\partial X)$ holds $\mu$-almost surely by Theorem \ref{thm:divergence type implies maximal critical exponent}, which leads to a contradiction  as $\dim_\mathrm{H}(\partial X)>0$. 
\end{proof}

\begin{proof}[Proof of Corollary \ref{cor:IRS of divergence type iff geodesic flow is ergodic}]
Let $G$ be a rank-one simple Lie group and $X$ be the associated Gromov hyperbolic symmetric space.   It is known \cite{sullivan1979density, roblin2003ergodicite} that the action of any discrete subgroup $\Gamma$ on $\partial^2 X$ is ergodic with respect to the Lebesgue measure class if and only if $\mathcal{P}_\Gamma(\dim_\mathrm{H}(\partial X)) = \infty$.  Moreover the geodesic flow on $\Gamma \backslash G$ is ergodic if and only if  $\Gamma$ acts ergodically on $\partial^2 X$. In particular, the ergodicity of the geodesic flow on $\Gamma \backslash G$ implies in general that $\Gamma$ has divergence type. The converse implication of the corollary follows from part (2) of Theorem \ref{thm:main theorem}.
\end{proof}

\begin{proof}[Proof of Corollary \ref{cor:Kestens theorem for IRS in surfaces}]
Let $G$ be a rank-one simple Lie group and $\mu$ be a non-atomic    invariant random subgroup   of $G$ so that $\mu$-almost every subgroup is torsion-free. We know that $\mu$-almost every subgroup is discrete and non-amenable, see e.g. Theorem \ref{thm:borel density theorem for IRS} and the remarks following Theorem \ref{thm:divergence at one half of critical exponent}. Our main result implies that $\delta(\mu) > \frac{1}{2}\mathrm{dim}_\text{H}(\partial X)$. The corollary follows from the precise relationship between the two quantities $\delta(\Gamma)$ and $\lambda_0(\Gamma \backslash X)$ mentioned in the introduction.
\end{proof}

\begin{proof}[Proof of Corollary \ref{cor:divergence type with T}] 
Let $G$ be one of the groups mentioned in the statement  and assume that $G$ has Kazhdan's property $(T)$. It is known by \cite[Theorem 4.4]{Cor}, \cite[Theorem 2]{leuzinger2004critical} and \cite[Corollary 1.4]{CDS} that there is some $\delta_c > 0$ such that every discrete subgroup $\Gamma$ of $G$ is either a lattice or satisfies $\delta(\Gamma) < \dim_\mathrm{H}(\partial X) - \delta_c$. 

Now let $\mu$ be a discrete   invariant random subgroup of $G$ of divergence type. Theorem \ref{thm:main theorem} implies that $\delta(\mu) = \dim_\mathrm{H}(\partial X)$. We deduce that $\mu$-almost every closed subgroup of $G$ is a lattice. Since $\mu$ is ergodic it must essentially be  the invariant random subgroup associated to some particular lattice $\Gamma$ of $G$, see \cite[Corollary 3.2]{SZ} for Lie groups or \cite[Corollary 5.6]{HT} in general.
\end{proof}

\begin{proof}[Proof of Corollary \ref{cor:application to pmp actions}]
Let $(Z,\mu)$ be a Borel space admitting a probability measure preserving action of $G$. We obtain an invariant random subgroup $\overline{\mu}$ of $G$ by considering the $G$-equivariant stabilizer map
$$ \mathrm{Stab} : Z \to \mathrm{Sub}(G), \quad z \mapsto G_z $$
and taking $\overline{\mu} = \mathrm{Stab}_* \mu$. The corollary follows by applying Theorem \ref{thm:main theorem} to $\overline{\mu}$, provided that the stabilizer of $\mu$-almost every point $ z \in Z$ is assumed  to be discrete and infinite.
\end{proof}

\begin{proof}[Proof of Corollary \ref{cor:application to pmp actions of discrete groups}]
Let $\Gamma$ be a lattice in the isometry group $\mathrm{Isom}(X)$ of some proper Gromov hyperbolic space  space $X$. Let  $(Z, \mu)$ be a Borel $\Gamma$-space with an invariant probability measure. 
It is possible to induce the $\Gamma$-action from $Z$ to a probability measure preserving action of $G = \mathrm{Isom}(X)$ on some Borel space $(\overline{Z}, \overline{\mu})$. See \cite[p. 75]{Zimmer} for details. The stabilizer $G_z$ of $\overline{\mu}$-almost every point $ z \in \overline{Z}$ is conjugate to the stabilizer $\Gamma_{z'}$ of some point in $z' \in Z$ and is in particular discrete. The result  follows from Corollary \ref{cor:application to pmp actions}.
\end{proof}



\appendix
\section{Measurability in the Chabauty Borel space}
\label{appendix}

This appendix is dedicated to various technical issues related to the Borel structure on $\mathrm{Sub}(G)$. In particular, we show that several notions and constructions used above, such as the critical exponent and  quasiconformal densities, all depend measurably on the choice of a discrete subgroup.

In general, given a second countable locally compact group $G$   and   an open subset $U \subset G$, the set $ \{H \le G \: : \: H \cap U \neq \emptyset \}$ is Chabauty open.  The Borel structure on $\mathrm{Sub}(G)$ coming from the Chabauty topology is called the \emph{Effros Borel structure} \cite[\S12.C]{Kechris}. Given a countable basis  $U_n$ for the topology of $G$, the Effros Borel structure is generated by the subsets $\{H \le G \: : \: H \cap U_n = \emptyset \}$. 

The first result of the appendix holds true for any second countable locally compact group $G$.

\begin{prop}
The subset $\mathrm{DSub}(G)$ of discrete subgroups in $\mathrm{Sub}(G)$ is Borel.
\end{prop}

\begin{proof}
This follows immediately from the observation that
$$ \mathrm{DSub}(G) = \bigcap_{\{n \in \mathbb{N}\::\: e \in U_n\} } \{ H \le G \: : \: H \cap (U_n \setminus \{e\}) = \emptyset \}. $$
\end{proof}

From now on we specialize the discussion to the situation where $X$ is a proper geodesic Gromov hyperbolic metric  space with   basepoint $o \in X$ and $G$ is a closed subgroup of $\mathrm{Isom}(X)$. Some of the notations that we use below depend implicitly on $o$.

\begin{prop}
\label{prop:orbital counting is Borel}
For every radius $R > 0$ the function
$$n_R : \mathrm{DSub}(G) \to \mathbb{N}, \quad n_R(\Gamma) = | \Gamma o \cap B_X(o,R) | $$
is Borel.
\end{prop}
\begin{proof}
For every   $\varepsilon > 0$ choose an arbitrary $\varepsilon$-separated and $2\varepsilon$-covering subset $S_\varepsilon$ of $X$ containing the base-point  $o$.  For every radius $R > 0$  denote
$$ n_{R,\varepsilon}(\Gamma) = |\{ x \in S_\varepsilon \: : \: \text{$ x \in B_X(o,R)$ and $B_{X}(x, 2\varepsilon)\cap \Gamma o \neq \emptyset $} \}| $$
where $\Gamma \in \mathrm{DSub}(G)$ is a discrete subgroup.
It follows from the definition of the Effros Borel structure that the function $n_{R,\varepsilon}$ is Borel for every $R,\varepsilon>0$. Note that
$$ n_R(\Gamma) = \limsup_{\varepsilon \to 0}n_{R,\varepsilon}(\Gamma) $$
and therefore $n_R$ is Borel as well for every $R > 0$.
\end{proof}

\begin{prop}
\label{prop:critical exponent is measurable, conjugation invariant}
The critical exponent function $$ \delta : \mathrm{DSub}(G) \to \left[0, \dim_\mathrm{H}(\partial X) \right] $$
is  Borel and conjugation invariant.
\end{prop}
\begin{proof}
By the previous proposition, the orbital counting function $n_R(\Gamma) = | \Gamma o \cap B_X(o,R) | $ is Borel on $\mathrm{DSub}(G)$ for every $R > 0$.
 It follows that the critical exponent 
$$ \delta(\Gamma) = \liminf_{R \to \infty} \frac{1}{R} \ln n_R(\Gamma) $$
is a Borel measurable function on $\mathrm{DSub}(G)$ as well.
The fact that $\delta$ is conjugation invariant is well known. To see this, it suffices to observe that
$$ |\Gamma o \cap B_X(o,R)| = |(g \Gamma g^{-1}) go \cap B_X(go, R)| $$
and recall that the  critical exponent is independent of the choice of basepoint.
\end{proof}

 Let $\mathcal{M}(X \cup \partial X)$ denote the  convex space of probability measures on $X \cup \partial X$ with the weak-$*$ topology. Given $\nu \in \mathcal{M}(X \cup \partial X)$ denote $\|\nu\| = \nu(X \cup \partial X)$. Let $\mathcal{M}_1(X \cup \partial X)$ denote the compact convex subset of $\mathcal{M}(X \cup \partial X)$ consisting of probability measures, that is all $\nu$ with $\|\nu\| = 1$.  
Given any real number $s \ge 0$ denote
$$ \mathrm{DSub}_{>s}(G) = \delta^{-1}((s,\infty)) = \{ \Gamma \in \mathrm{DSub}(\Gamma) \: : \: \delta(\Gamma) > s \}.$$
In particular $\mathrm{DSub}_{>s}(G)$  is a  Borel subset by the previous proposition.

\begin{prop}
\label{prop:s-measures are measurable}
For every real number $s \ge 0$  the map
$$ W_{s} : \mathrm{DSub}_{>s}(G) \to \mathcal{M}_1(X \cup \partial X), \quad W_s : \Gamma \mapsto \frac{\sum_{g\in \Gamma} e^{-s d_{X}(go,o)} \delta_{go} }{\sum_{g\in \Gamma}e^{-sd_{X}(go,o)}}$$
is Borel.
\end{prop}
\begin{proof}
We proceed similarly to Proposition \ref{prop:orbital counting is Borel}. Namely, for every $\varepsilon$  choose an arbitrary $\varepsilon$-separated and $2\varepsilon$-covering subset $S_\varepsilon$ of $X$ containing the basepoint  $o$. Denote $S_\varepsilon(\Gamma) = \{x \in S_\varepsilon\::\: B_{X}(x,)) \cap \Gamma \neq \emptyset\}$ and
$$ W_{s,\varepsilon}(\Gamma) =  \frac{\sum_{x \in S_\varepsilon(\Gamma) } e^{-s d_X(x,o)} \delta_x }{\sum_{x \in S_\varepsilon(\Gamma)} e^{-s d_X(x,o)} } \in \mathcal{M}_1(X \cup \partial X)$$
for every discrete subgroup $\Gamma \in \mathrm{DSub}_{>s}(G)$. By the definition of the Effros Borel structure  the function $W_{s,\varepsilon}$ is Borel for all $\varepsilon$. Note that for every discrete subgroup $\Gamma \in \mathrm{DSub}_{>s}(G)$ the weak-$*$ limit of the probability measures $W_{s,\varepsilon}(\Gamma)$ exists and in fact
$$ W_{s}(\Gamma) = \lim_{\varepsilon \searrow 0} W_{s,\varepsilon} (\Gamma). $$
The function $W_s$ is therefore Borel, as required.
\end{proof}

Recall that the standard construction of a quasiconformal density for a fixed discrete group $\Gamma$ of divergence type involves taking a weak-$*$ accumulation point as $s \searrow \delta(\Gamma)$. The following general lemma is needed to be able to do so in a measurable way.

\begin{lemma}\label{measureablechoice}
Let $X$ be a Borel space, $M$ a second countable compact metric space and $f_n : X \to M$ a sequence of Borel maps. Then there is a Borel map $p : X \to M$ such that $p(x)$ is an accumulation point in $M$ of the sequence $f_n(x)$ for every $x \in X$.
\end{lemma}
\begin{proof}
Let $K(M)$ denote the Hausdorff space of all closed subsets of $M$. Let $P(x) \in K(M)$ be the non-empty  closed subset consisting of all the accumulation points of the sequence $f_n(x)$ for every $x \in X$. We claim that the map $P : X \to K(M)$ is Borel. The existence of the required selection function $p(x)$ will then follow at once   from the Kuratowski--Ryll-Nardzewski selection theorem \cite[Theorem 12.13]{Kechris}.

There is a countable family of open subsets $N_m \subset M$ for $m \in \mathbb{N}$ such that the Borel structure on $K(M)$ is generated by the Borel subsets $ \{ \text{$C \subset M$ closed} \: : \: C \cap N_m = \emptyset \}$. Note that $P(x) \cap N_m = \emptyset$ is equivalent to $f_n(x) \in M \setminus N_m$ for all $n$ sufficiently large, which is a Borel condition.
\end{proof}

Assume that $G$ is acting cocompactly on $X$ and admits a uniform lattice. We conclude   appendix \ref{appendix} by showing that  a family of  quasiconformal densities can be chosen in measurable way for all the discrete subgroups of a given critical exponent and having divergence type. 
Given any real number $\delta \ge 0 $ denote
$$ \mathrm{DSub}_{\text{$\delta$-diver}}(G) = \{ \Gamma \in \mathrm{DSub}(\Gamma) \: : \: \text{$\delta(\Gamma) = \delta$ and $\Gamma$ is of divergence type} \}. $$

\begin{prop}
\label{prop:existance of a measurable quasiconformal density}
For every $\delta \ge 0$ there is a constant $ d \ge 1$ and a Borel map
$$ \nu : \mathrm{DSub}_{\text{$\delta$-diver}}(G) \times X \to \mathcal{M}(  \partial X), \quad \nu : (\Gamma, x) \mapsto \nu_x^\Gamma $$ 
so that $\{\nu^\Gamma_x\}_{x\in X}$ is a $\Gamma$-quasiconformal density of dimension $\delta$ and distortion $d$. Moreover we may normalize  so that $\|\nu^\Gamma_o\| = 1$ for all $\Gamma \in \mathrm{DSub}_{\text{$\delta$-diver}}(G)$.
\end{prop}


\begin{proof}
Consider the Borel maps $W_s$   constructed in Proposition \ref{prop:s-measures are measurable} for every $s > \delta$. We may now apply Lemma \ref{measureablechoice} to obtain 
$$\nu_o : \mathrm{DSub}_{\text{$\delta$-diver}}(G) \to \mathcal{M}(X \cup \partial X)$$ which is a Borel map assigning to each discrete subgroup $\Gamma \in \mathrm{DSub}_{\text{$\delta$-diver}}(G)$ some weak-$*$ accumulation point of the probability measures $W_s(\Gamma)$ as $s \searrow \delta$. We may now extend $\nu$ to a map defined on $\mathrm{DSub}_{\text{$\delta$-diver}}(G) \times X$ in a Borel measurable manner   by insisting that the relation
$$\frac{d\nu^{\Gamma}_{x}}{d\nu^{\Gamma}_{o}}(\xi)=e^{\delta \beta_{\xi}(o,x)}$$
should be  satisfied for every point $x \in X$.
The collection $\{\nu^{\Gamma}_{x}\}_{x\in X}$ is indeed a $\Gamma$-quasiconformal density of critical exponent $\delta$ and distortion $ d \ge 1$. These parameters depend only on the hyperbolicity constant of the Gromov space $X$ for every discrete subgroup $\Gamma \in \mathrm{DSub}_{\text{$\delta$-diver}}(G)$. See  e.g.  \cite[Lemma 6.1]{MYJ}. 

\end{proof}

\bibliography{critical}

\newcommand{\etalchar}[1]{$^{#1}$}
\begin{thebibliography}{ABB{\etalchar{+}}17}

\bibitem[ABB{\etalchar{+}}17]{7S}
Miklos Abert, Nicolas Bergeron, Ian Biringer, Tsachik Gelander, Nikolay
  Nikolov, Jean Raimbault, and Iddo Samet.
\newblock On the growth of {$L^2$}-invariants for sequences of lattices in
  {L}ie groups.
\newblock {\em Annals of math}, 185(3):711--790, 2017.

\bibitem[AC18]{Arzh-Cash}
Goulnara Arzhantseva and Christopher Cashen.
\newblock Cogrowth for group actions with strongly contracting elements.
\newblock {\em arXiv preprint arXiv:1803.05782}, 2018.

\bibitem[AGV14]{abert2014kesten}
Mikl{\'o}s Ab{\'e}rt, Yair Glasner, and B{\'a}lint Vir{\'a}g.
\newblock Kesten's theorem for invariant random subgroups.
\newblock {\em Duke Mathematical Journal}, 163(3):465--488, 2014.

\bibitem[BJ97]{BJ}
Christopher Bishop and Peter Jones.
\newblock Hausdorff dimension and {K}leinian groups.
\newblock {\em Acta Mathematica}, 179(1):1--39, 1997.

\bibitem[BK02]{BK}
Mario Bonk and Bruce Kleiner.
\newblock Rigidity for quasi-{M}oebius actions.
\newblock {\em Journal of Differential Geometry}, 61(1):81--106, 2002.

\bibitem[BN13]{Bowen-Nevo-negcurved}
Lewis Bowen and Amos Nevo.
\newblock von-{N}eumann and {B}irkhoff ergodic theorems for negatively curved
  groups.
\newblock {\em arXiv preprint arXiv:1303.4109}, 2013.

\bibitem[Can14]{Cannizzo}
Jan Cannizzo.
\newblock {\em Schreier Graphs and Ergodic Properties of Boundary Actions}.
\newblock PhD thesis, Universit{\'e} d'Ottawa/University of Ottawa, 2014.

\bibitem[CDS17]{CDS}
R{\'e}mi Coulon, Fran{\c{c}}oise Dal'bo, and Andrea Sambusetti.
\newblock Growth gap in hyperbolic groups and amenability.
\newblock {\em arXiv preprint arXiv:1709.07287}, 2017.

\bibitem[Coo93]{Coornaert}
Michel Coornaert.
\newblock Mesures de {P}atterson-{S}ullivan sur le bord d\'un espace
  hyperbolique au sens de {G}romov.
\newblock {\em Pacific Journal of Mathematics}, 159(2):241--270, 1993.

\bibitem[Cor90]{Cor}
Kevin Corlette.
\newblock Hausdorff dimensions of limit sets {I}.
\newblock {\em Inventiones mathematicae}, 102(1):521--541, 1990.

\bibitem[DSU17]{DSU}
Tushar Das, David Simmons, and Mariusz Urba{\'n}ski.
\newblock {\em Geometry and dynamics in {G}romov hyperbolic metric spaces},
  volume 218.
\newblock American Mathematical Soc., 2017.

\bibitem[Els73a]{elstrodt1973resolvente}
J{\"u}rgen Elstrodt.
\newblock Die resolvente zum eigenwertproblem der automorphen formen in der
  hyperbolischen ebene. teil i.
\newblock {\em Mathematische Annalen}, 203(4):295--330, 1973.

\bibitem[Els73b]{elstrodt1973resolvente2}
J{\"u}rgen Elstrodt.
\newblock Die resolvente zum eigenwertproblem der automorphen formen in der
  hyperbolischen ebene. teil ii.
\newblock {\em Mathematische Zeitschrift}, 132(2):99--134, 1973.

\bibitem[Els74]{elstrodt1974resolvente3}
J{\"u}rgen Elstrodt.
\newblock Die resolvente zum eigenwertproblem der automorphen formen in der
  hyperbolischen ebene. teil iii.
\newblock {\em Mathematische Annalen}, 208(2):99--132, 1974.

\bibitem[GKN12]{GKN}
Rostislav Grigorchuk, Vadim Kaimanovich, and Tatiana Nagnibeda.
\newblock Ergodic properties of boundary actions and the {N}ielsen--{S}chreier
  theory.
\newblock {\em Advances in Mathematics}, 230(3):1340--1380, 2012.

\bibitem[GL17]{GL}
Tsachik Gelander and Arie Levit.
\newblock Invariant random subgroups over non-{A}rchimedean local fields.
\newblock {\em arXiv preprint arXiv:1707.03578}, 2017.

\bibitem[HT16]{HT}
Yair Hartman and Omer Tamuz.
\newblock Stabilizer rigidity in irreducible group actions.
\newblock {\em Israel Journal of Mathematics}, 216(2):679--705, 2016.

\bibitem[Jae15]{Jaerisch}
Johannes Jaerisch.
\newblock A lower bound for the exponent of convergence of normal subgroups of
  {K}leinian groups.
\newblock {\em The Journal of Geometric Analysis}, 25(1):298--305, 2015.

\bibitem[Kec12]{Kechris}
Alexander Kechris.
\newblock {\em Classical descriptive set theory}, volume 156.
\newblock Springer Science \& Business Media, 2012.

\bibitem[Kes59]{kesten1959symmetric}
Harry Kesten.
\newblock Symmetric random walks on groups.
\newblock {\em Transactions of the American Mathematical Society},
  92(2):336--354, 1959.

\bibitem[Leu04]{leuzinger2004critical}
Enrico Leuzinger.
\newblock Critical exponents of discrete groups and $l^2$--spectrum.
\newblock {\em Proceedings of the American Mathematical Society},
  132(3):919--927, 2004.

\bibitem[Mat05]{Mat}
Katsuhiko Matsuzaki.
\newblock Isoperimetric constants for conservative {F}uchsian groups.
\newblock {\em Kodai Mathematical Journal}, 28(2):292--300, 2005.

\bibitem[MYJ15]{MYJ}
Katsuhiko Matsuzaki, Yasuhiro Yabuki, and Johannes Jaerisch.
\newblock Normalizer, divergence type and {P}atterson measure for discrete
  groups of the {G}romov hyperbolic space.
\newblock {\em arXiv preprint arXiv:1511.02664}, 2015.

\bibitem[{\=O}sh02]{Ohshika}
Kenʼichi {\=O}shika.
\newblock {\em Discrete groups}, volume 207.
\newblock American Mathematical Soc., 2002.

\bibitem[Pat76]{patterson1976limit}
Samuel Patterson.
\newblock The limit set of a {F}uchsian group.
\newblock {\em Acta mathematica}, 136(1):241--273, 1976.

\bibitem[Pat77]{Pat2}
Samuel Patterson.
\newblock Spectral theory and {F}uchsian groups.
\newblock In {\em Mathematical Proceedings of the Cambridge Philosophical
  Society}, volume~81, pages 59--75. Cambridge University Press, 1977.

\bibitem[Pat83]{Pat}
Samuel Patterson.
\newblock Further remarks on the exponent of convergence of {P}oincar\'e
  series.
\newblock {\em Tohoku Math Journal}, 35(2):357--373., 1983.

\bibitem[Qui06]{Quint}
Jean-Fran{\c{c}}ois Quint.
\newblock An overview of {P}atterson-{S}ullivan theory.
\newblock In {\em Workshop The barycenter method, FIM, Zurich}, 2006.

\bibitem[Rob03]{roblin2003ergodicite}
Thomas Roblin.
\newblock {\em Ergodicit{\'e} et {\'e}quidistribution en courbure
  n{\'e}gative}.
\newblock Soci{\'e}t{\'e} math{\'e}matique de France, 2003.

\bibitem[Sul79]{sullivan1979density}
Dennis Sullivan.
\newblock The density at infinity of a discrete group of hyperbolic motions.
\newblock {\em Publications Math{\'e}matiques de l'Institut des Hautes
  {\'E}tudes Scientifiques}, 50(1):171--202, 1979.

\bibitem[Sul87]{sullivan1987related}
Dennis Sullivan.
\newblock Related aspects of positivity in {R}iemannian geometry.
\newblock {\em Journal of differential geometry}, 25(3):327--351, 1987.

\bibitem[SZ94]{SZ}
Garrett Stuck and Robert~J Zimmer.
\newblock Stabilizers for ergodic actions of higher rank semisimple groups.
\newblock {\em Annals of Mathematics}, pages 723--747, 1994.

\bibitem[Zim13]{Zimmer}
Robert~J Zimmer.
\newblock {\em Ergodic theory and semisimple groups}, volume~81.
\newblock Springer Science \& Business Media, 2013.

\end{thebibliography}
\bibliographystyle{alpha}

\end{document}